\documentclass[11pt]{amsart}

\usepackage{amsfonts,amsmath,amssymb}

\usepackage{enumerate}

\newtheorem{theorem}{Theorem}[section]

\newtheorem{lemma}[theorem]{Lemma}

\newtheorem{proposition}[theorem]{Proposition}

\newtheorem{definition}{Definition}[section]

\newtheorem{corollary}[theorem]{Corollary}

\newtheorem{remark}[theorem]{Remark}

\newcommand{\cl}[1]{\mathcal{#1}} 

\newcommand{\bb}[1]{\mathbb{#1}}

\newcommand\id{\mathop{\rm id}}

\newcommand\ran{\mathop{\rm Ran}}

\newcommand\supp{\mathop{\rm supp}}

\newcommand\vn{\mathop{\rm VN}}

\newcommand{\Ref}[1]{\mathrm{Ref}(#1)} 

\newcommand\nph{\varphi}

\begin{document}

\title{Operator synthesis and tensor products}

\author{G.K. Eleftherakis and I.G. Todorov}

\address{Department of Mathematics, University of Athens, 
Athens 157 84, Greece}

\email{gelefth@math.uoa.gr}

\address{Pure Mathematics Research Centre, Queen's University Belfast,
Belfast BT7 1NN, United Kingdom}

\email{i.todorov@qub.ac.uk}

\date{12 January 2013}

\begin{abstract}
We show that Kraus' property $S_{\sigma}$ is preserved under 
taking weak* closed sums with masa-bimodules of finite width, 
and establish an intersection formula for 
weak* closed spans of tensor products, one of whose terms 
is a masa-bimodule of finite width. We initiate the study of 
the question of when operator synthesis is preserved under 
the formation of products and
prove that 
the union of finitely many sets of the form $\kappa\times\lambda$,
where $\kappa$ is a set of finite width, while $\lambda$ is 
operator synthetic, is, under a necessary restriction on the sets $\lambda$, 
again operator synthetic. 
We show that property $S_{\sigma}$ is preserved under
spatial Morita subordinance.
En route, we prove 
that non-atomic ternary masa-bimodules possess 
property $S_{\sigma}$ hereditarily.
\end{abstract}

\maketitle

\section{Introduction}

Operator synthesis was introduced by W.B. Arveson in his seminal paper 
\cite{a} as an operator theoretic version of the notion of spectral synthesis in 
Harmonic Analysis, and
was subsequently developed by 
J. Froelich, A. Katavolos,  J. Ludwig, V.S. Shulman, N. Spronk, L. Turowska
and the authors \cite{eletod}, \cite{f}, \cite{kt}, \cite{st}, \cite{st2}, \cite{spt},
\cite{jlms}, \cite{houston}, among others.
It was shown in \cite{f}, \cite{spt} and \cite{lt} that, 
for a large class of locally compact groups $G$,
given a closed subset $E$ of $G$, there is a canonical way to
produce a subset $E^*$ of the direct product $G\times G$, so that
the set $E$ satisfies spectral synthesis if and only if the set
$E^*$ satisfies operator synthesis. Thus, the well-known, and still
open, problem of whether the union of two sets of spectral synthesis satisfies
spectral synthesis can be viewed as a special case of the problem asking
whether the union of two operator synthetic sets is operator
synthetic.

Another problem in Harmonic Analysis asks 
when the product of two sets of spectral synthesis 
is again synthetic. 
The analogous question in the 
operator theory setting  
is closely related to 
\emph{property $S_{\sigma}$}, introduced by J. Kraus in 
\cite{kraus_tams}.
It is widely recognised that 
functional analytic tensor products display a larger degree of subtlety
than the algebraic ones, the reason for this being the fact that
they are defined as the \emph{completion} of the 
\emph{algebraic} tensor product of two objects (say, operator algebras, 
or operator spaces) with respect to an appropriate topology. 
Therefore,
it is usually not an easy task to determine the intersection
of two spaces, both given as completed
tensor products.
Such issues give rise to a number of 
important concepts in Operator Algebra Theory, {\it e.g.}
exactness \cite{pisier}.
Property $S_{\sigma}$ is instrumental in describing 
such intersections, and is closely related to a number of 
important approximation properties.
In particular, it was shown in \cite{kraus} to be equivalent to the
\emph{$\sigma$-weak approximation property},
while in \cite{hk}, an equivalence of when 
a group von Neumann algebra $\vn(G)$ possesses property $S_{\sigma}$ 
was formulated in terms of an approximation property of the underlying group $G$.

In this paper, we initiate the study of the question of when the 
direct product of two operator synthetic sets is operator synthetic. 
Furthermore, we combine the two stands of investigation 
highlighted in the previous two paragraphs by 
studying the question of when the union of direct products of 
operator synthetic sets is operator synthetic. 
The setting of operator synthesis is provided by the theory of 
masa-bimodules (see \cite{a}, \cite{eks}, \cite{st} and \cite{st2}).
A prominent role in our study is played by the 
masa-bimodules of \emph{finite width}.
This class is a natural extension of the 
class of CSL algebras of finite width, which was introduced in \cite{a}
as a far reaching, yet tractable, generalisation of \emph{nest algebras} \cite{dav}. 
It was shown in \cite{hok} that CSL algebras of finite width 
possess property $S_{\sigma}$. 
However, this class has for long remained the main example of 
operator spaces known to have this property. 
We note that the question of whether every 
weak* closed masa-bimodule possesses property $S_{\sigma}$ 
is still open (see \cite{kraus}).

It should be noted that masa-bimodules of finite width
have been studied in a number of other contexts. 
They include as a subclass the masa-bimodules which are 
\emph{ternary rings of operators} \cite{houston}, 
a class of operator spaces that has been studied extensively 
for the purposes of Operator Space Theory \cite{bm}. 
The supports of masa-bimodules of finite width 
(called henceforth \emph{sets of finite width}) 
are precisely the sets of solutions 
of systems of inequalities, and were shown 
in \cite{st} and \cite{jlms} to be operator synthetic, 
providing in this way the largest single class of sets 
that are known to satisfy operator synthesis. 
It was shown in \cite{eletod} that the 
union of an operator synthetic set and a set of finite width is 
operator synthetic.
In \cite{stt}, this line of investigation was continued 
by showing that masa-bimodules of finite width 
satisfy a rank one approximation property, 
and a large class of examples of 
\emph{sets of operator multiplicity}
was exhibited within this class.
They were the motivating example for the introduction and study of 
$\mathfrak{I}$-decomposable masa-bimodules in \cite{eletod}.

The weak* closed masa-bimodules are
precisely the weak* closed invariant subspaces of 
\emph{Schur multipliers} or, equivalently, of 
weak* continuous (completely) bounded masa-bimodule maps.
The projections in the algebra of all Schur multipliers, 
called henceforth \emph{Schur idempotents}, were 
at the core of the methods developed in 
\cite{eletod} in order to address the 
union problem, as well as the closely related 
problem of the reflexivity of weak* closed spans.

Here we significantly extend the techniques 
whose development was initiated in \cite{eletod}
by establishing an intersection formula involving tensor products 
and applying it to the study of the product and union problems
described above. 
Simultaneously, we initiate the study of the question of whether 
property $S_{\sigma}$ is preserved under taking weak* closed spans. 

The paper is organised as follows. 
After gathering some preliminary notions and 
results in Section \ref{s_prel}, we address in Section \ref{s_stab}
the preservation problem for property $S_{\sigma}$ outlined in the previous paragraph,
showing that 
the class of spaces possessing $S_{\sigma}$ is closed under taking 
weak* closed sums with masa-bimodules of finite width (Theorem \ref{th_sump}).
As a consequence, the weak* closed span of any finite number of 
masa-bimodules of finite width possesses property $S_{\sigma}$.
En route, we give a sufficient condition for a ternary masa-bimodule
to possess $S_{\sigma}$ hereditarily (Theorem \ref{th_trohed}).

In Section \ref{s_insum}, we establish the 
intersection formula 
\begin{equation}\label{eq_inf}
\bigcap_{j_1,\dots,j_r} \overline{\cl B_{j_1}^1 \otimes \cl U_1 + \cdots + \cl B_{j_r}^r \otimes \cl U_r}
 = \overline{(\cap_{j_1}\cl B_{j_1}^1) \otimes \cl U_1 + 
 \cdots + (\cap_{j_r}\cl B_{j_r}^1) \otimes \cl U_r},
\end{equation}
 valid for all masa-bimodules $\cl B_{j_p}^p$ of finite width
 and all weak* closed spaces of operators $\cl U_p$, $p = 1,\dots,r$, 
 $j_p = 1,\dots,m_p$ (Corollary \ref{epsilona}). 
In Section \ref{s_appos}, we formalise the relation between property $S_{\sigma}$ 
and the problem for the synthesis of products (see Corollary \ref{c_synps}). 
As part of Proposition \ref{1234}, we establish a subspace version of the 
relation between \emph{Fubini products} and the 
\emph{algebra tensor product formula} discussed in \cite{kraus_tams}.
These results, along with the formula (\ref{eq_inf}), are used
to show that 
the union of finitely many products $\kappa_i\times\lambda_i$,
where the sets $\kappa_i$ are of finite width, 
while $\lambda_i$ are operator synthetic sets satisfying certain necessary 
restrictions, is operator synthetic (Theorem \ref{th_fws}). 

Finally, in Section \ref{s_mor}, we show that property $S_{\sigma}$ is preserved 
under spatial Morita subordinance. As a corollary, we 
obtain that if $\cl L_1$ and $\cl L_2$ are isomorphic CSL's then 
the CSL algebra ${\rm Alg}\cl L_1$ possesses property $S_{\sigma}$
if and only if ${\rm Alg}\cl L_2$ does so.

It is natural to wonder whether our results are valid for the 
more general class consisting of intersections of 
$\mathfrak{I}$-decomposable spaces introduced in \cite{eletod}. 
We note that this class contains properly the class of 
masa-bimodules of finite width.
Progress in this direction would rely on the answer of the question of whether 
the \emph{approximately $\mathfrak{I}$-injective} masa-bimodules 
(that is, the intersections of descending sequences of ranges of uniformly bounded 
Schur idempotents) satisfy property $S_{\sigma}$;
this question, however, is still open.


\section{Preliminaries}\label{s_prel}

In this section, we collect some preliminary notions and results that will be needed in the sequel. 
If $H$ and $K$ are Hilbert spaces, we denote by $\cl B(H,K)$ the space of all bounded linear operators
from $H$ into $K$, and write $\cl B(H) = \cl B(H,H)$. The space $\cl B(H,K)$ is the dual of 
the ideal of all trace class operators from $K$ into $H$, and can hence be endowed with a weak* topology; 
we note that this is the weakest topology on $\cl B(H,K)$ with respect to which the functionals $\omega$ of the form 
$$\omega(T) = \sum_{k=1}^{\infty} (T\xi_k,\eta_k), \ \ \ T\in \cl B(H,K),$$
where $(\xi_k)_{k\in \bb{N}}\subseteq H$ and $(\eta_k)_{k\in \bb{N}}\subseteq K$ are square summable sequences of vectors, 
are continuous. 
In the sequel, we denote by $\overline{\cl U}$ the weak* closure of a set $\cl U\subseteq \cl B(H,K)$. 

Throughout the paper, $H,K,H_1,K_1,H_2$ and $K_2$ will denote Hilbert spaces. Let $\cl V\subseteq \cl B(H_1,H_2)$
and $\cl U\subseteq \cl B(K_1,K_2)$ be weak* closed subspaces. 
We denote by 
$\cl V\bar\otimes\cl U$ the weak* closed subspace of $\cl B(H_1\otimes K_1,H_2\otimes K_2)$ 
generated by the operators of the form $S\otimes T$, where $S\in \cl V$ and $T\in \cl U$.
Here,  $H\otimes K$ is the Hilbertian tensor product of
$H$ and $K$, and we use the natural identification 
$$\cl B(H_1\otimes K_1,H_2\otimes K_2) \equiv \cl B(H_1,H_2)\bar\otimes\cl B(K_1,K_2).$$

We will use some basic notions from Operator Space Theory; we refer the reader to 
the monographs \cite{bm}, \cite{er}, \cite{paul} and \cite{pisier} for the relevant definitions. 
If $\cl X$ is a linear space, we denote by $\id$ the identity map on $\cl X$. 
The range of a linear map $\phi$ on $\cl X$ is denoted by $\ran\phi$. 
As customary, the map $\phi$ is called idempotent if $\phi\circ\phi = \phi$;
we let $\phi^{\perp} = \id - \phi$. 
If $\cl X_1$ and $\cl X_2$ are subspaces of $\cl X$, we set 
$\cl X_1 + \cl X_2 = \{x_1 + x_2 : x_i\in \cl X_i, i = 1,2\}$.
If $\cl V_i \subseteq \cl B(H_1,H_2)$ and $\cl U_i \subseteq \cl B(K_1,K_2)$ are weak* closed subspaces, $i = 1,2$,
and $\phi : \cl V_1\to \cl V_2$ and $\psi : \cl U_1\to \cl U_2$ are
completely bounded weak* continuous maps, then there exists a (unique)
completely bounded weak* continuous map 
$\phi\otimes\psi : \cl V_1\bar\otimes\cl U_1\to \cl V_2 \bar\otimes\cl U_2$ 
such that $\phi\otimes\psi(A\otimes B) = \phi(A)\otimes\psi(B)$, 
$A\in \cl V_1$, $B\in \cl U_1$ \cite{bm}. 
In the case $\cl U_1 = \cl U_2 = \cl B(K_1,K_2)$, we write throughout the paper $\tilde{\phi} = \phi\otimes\id$. 
We denote by $\cl V_*$ the space of all weak* continuous functionals on $\cl V$. 
If $\cl V\subseteq \cl B(H_1,H_2)$ is a weak* closed subspace of operators
and $\omega\in \cl V_*$ then
we set $R_{\omega} = \tilde{\omega}$; thus, 
$R_{\omega} : \cl V\bar\otimes\cl B(K_1,K_2)\to \cl B(K_1,K_2)$
is  the \emph{Tomiyama's right slice map} corresponding to $\omega$
(here we have use the natural identification 
$\bb{C}\bar\otimes \cl B(K_1,K_2) \equiv \cl B(K_1,K_2)$).
We note that 
$R_{\omega}(A\otimes B) = \omega (A)B$, $A\in \cl V$, $B\in \cl B(K_1,K_2)$. 
If, further, $\cl U\subseteq \cl B(K_1,K_2)$ is a weak* closed subspace, the \emph{Fubini product}
$\cl F(\cl V, \cl U)$ of $\cl V$ and $\cl U$ is the 
subspace of $\cl V\bar\otimes \cl B(K_1,K_2)$ given by
$$\cl F(\cl V, \cl U)=\{T \in \cl V\bar\otimes \cl B(K_1,K_2): R_\omega (T)\in \cl U, \ \mbox{ for all } \omega\in \cl V_*\}.$$
If $\xi\in H_1$ and $\eta\in H_2$, we let $\omega_{\xi,\eta}$ be the vector functional on $\cl B(H_1,H_2)$
given by $\omega_{\xi,\eta}(A) = (A\xi,\eta)$, $A\in \cl B(H_1,H_2)$; 
we use the same symbol to denote the restriction of $\omega_{\xi,\eta}$ to the subspace $\cl U\subseteq \cl B(H_1,H_2)$.

It is easy to notice that $\cl V\bar\otimes \cl U \subseteq \cl F(\cl V,\cl U)$. 
The subspace $\cl V$ is said to possess \emph{property $S_\sigma$} 
if $\cl F(\cl V, \cl U) = \cl V \bar\otimes \cl U$ for all weak* subspaces $\cl U\subseteq \cl B(K_1,K_2)$ and all Hilbert spaces $K_1, K_2$. 
This notion was introduced by Kraus in \cite{kraus_tams}, 
where he showed that $\cl B(K_1,K_2)$ possesses property $S_{\sigma}$. 
(We note that Kraus considered the case $K_1 = K_2$; 
however, it is easy to see that one can state both the definition and 
the result in terms of two Hilbert spaces.)
From this fact, one can easily derive the formula
$$\cl F(\cl V, \cl U) = (\cl V\bar\otimes \cl B(K_1,K_2)) \cap (\cl B(H_1,H_2) \bar\otimes \cl U).$$

Now suppose that $H_1$ and $H_2$ are separable Hilbert spaces and 
$\cl D_1\subseteq \cl B(H_1)$ and $\cl D_2\subseteq \cl B(H_2)$ are 
maximal abelian selfadjoint algebras (for brevity, masas). 
A linear map $\phi$ on $\cl B(H_1,H_2)$ is called \emph{$\cl D_2,\cl D_1$-modular}, or a 
\emph{masa-bimodule map} when $\cl D_1$ and $\cl D_2$ are clear from the context, 
provided
$\phi(BXA) = B\phi(X)A$, for all $X\in \cl B(H_1,H_2)$, $A\in \cl D_1$ and $B\in \cl D_2$. 
We call the completely bounded weak* continuous $\cl D_2,\cl D_1$-modular maps on 
$\cl B(H_1,H_2)$ \emph{Schur maps} (relative to the pair $(\cl D_1,\cl D_2)$); 
a Schur map that is also an idempotent is called a
\emph{Schur idempotent}.
This terminology is natural 
in view of the fact that Schur maps correspond precisely to Schur multipliers, 
provided a particular coordinate representation of $\cl D_1$ and $\cl D_2$ is chosen. 
We refer the reader to \cite{eletod} for details; 
we will return to this perspective in Section \ref{s_appos}.

A \emph{$\cl D_2,\cl D_1$-bimodule}, or simply a \emph{masa-bimodule} when $\cl D_1$ and $\cl D_2$ are understood 
from the context, is a subspace $\cl V\subseteq \cl B(H_1,H_2)$ such that 
$BXA\in \cl V$ whenever $X\in \cl V$, $A\in \cl D_1$ and $B\in \cl D_2$. 
Masa-bimodules will 
be assumed to be weak* closed throughout the paper; 
they are precisely the weak* closed subspaces invariant under all Schur maps (see \cite[Proposition 3.2]{eletod}). 
A weak* closed masa-bimodule $\cl M$ is called \emph{ternary} 
\cite{kt}, \cite{houston} 
if $\cl M$ is a ternary ring of operators, that is, if $TS^*R\in \cl M$ whenever $T,S,R\in \cl M$ (see
also \cite{bm}). It is not difficult to see
that every ternary masa-bimodule is the intersection of a descending sequence 
of ranges of contractive Schur idempotents; this fact will be used extensively 
hereafter.
It is easy to notice that the ternary masa-bimodules acting on a single Hilbert space 
which are unital algebras are precisely the von Neumann algebras with abelian commutant.

A \emph{nest} on a Hilbert space $H$ is a totally ordered family of closed subspaces of $H$ 
that contains the intersection and the closed linear span (denoted $\vee$) of any if its subsets.
A \emph{nest algebra} is the subalgebra of $\cl B(H)$ of all operators leaving invariant each 
subspace of a given nest. 
A \emph{nest algebra bimodule} is a subspace $\cl V\subseteq \cl B(H_1,H_2)$ 
for which there exist nest algebras $\cl A\subseteq \cl B(H_1)$ and $\cl B\subseteq \cl B(H_2)$ 
such that $\cl B\cl V\cl A\subseteq \cl V$. All nest algebra bimodules will be assumed to be weak* closed. 
A $\cl D_2,\cl D_1$-bimodule $\cl V$ is said to have \emph{finite width} if it is of the form
$\cl V = \cl V_1\cap\cdots \cap\cl V_k$, where each 
$\cl V_j$ is a $\cl D_2,\cl D_1$-bimodule that is also a nest algebra bimodule. 
The smallest $k$ with this property is called the \emph{width} of $\cl V$. 
We note that every ternary masa-bimodule has width at most two \cite{kt}. 
If each $\cl V_j$ is a nest algebra then $\cl V$ is called a CSL algebra of finite width \cite{a}. 
It was shown in \cite{kraus_tams} that von Neumann algebras with  
abelian commutant possess property $S_{\sigma}$  and in \cite{hok} that every 
CSL algebra of finite width possesses property $S_{\sigma}$.

Suppose that $\cl V\subseteq \cl B(H_1,H_2)$ is a nest algebra bimodule. 
It was shown in \cite{ep} that there exist nests  
$\cl N_1\subseteq \cl B(H_1)$ and $\cl N_2 \subseteq \cl B(H_2)$ and an increasing
$\vee$-preserving map $\nph : \cl N_1\rightarrow \cl N_2$ 
such that
$$\cl V = \{X\in \cl B(H_1,H_2) : XN = \nph(N)XN, \ N\in \cl N_1\}.$$
Let $\{P_i\}_{i\in \bb{N}}\subseteq \cl N_1$ be a (countable) subset
dense in $\cl N_1$ in the strong operator topology 
such that the set $\{\nph(P_i)\}_{i\in \bb{N}}$ is
dense in $\cl N_2$, and 
$$\cl F_n = \{0,P_1,P_2,\dots,P_n,I\} = \{0 < N_1 < N_2 < \dots < N_n < I\}.$$
Set $N_0 = 0$ and $N_{n+1} = I$ and 
let $\phi_n,\psi_n : \cl B(H_1,H_2)\to \cl B(H_1,H_2)$ be the Schur idempotents 
(relative to any pair $(\cl D_1,\cl D_2)$ with $\cl N_1\subseteq \cl D_1$ qnd $\cl N_2\subseteq \cl D_2$) 
given by 
$$\phi_n(X) = \sum_{i=0}^n (\nph(N_{i+1}) - \nph(N_i))X(N_{i+1} - N_i), \ \ X\in \cl B(H_1,H_2),$$
$$\psi_n(X) = \sum_{0\leq i < j\leq n} (\nph(N_{i+1}) - \nph(N_i))X(N_{j+1} - N_j), \ \ X\in \cl B(H_1,H_2),$$
and $\cl M_n$ and $\cl W_n$ be the ranges of $\phi_n$ and $\psi_n$, respectively. 
We have that
$$\phi_n\psi_n = 0, \ \ \cl W_n\subseteq\cl V\subseteq \cl W_n + \cl M_n, \ \ \cl W_n\subseteq \cl W_{n+1}, \ \cl M_{n+1}\subseteq \cl M_n, \ n\in \bb{N},$$
and $\cap_{n=1}^{\infty} \cl M_n \subseteq \cl V$. 
We will call the family 
$(\phi_n, \psi_n, \cl M_n, \cl W_n)_{n\in \bb{N}}$
a \emph{decomposition scheme} for $\cl V$. 
Decomposition schemes were first explicitly used (although not referred to as such) 
in \cite{eletod} for the study of reflexivity and synthesis problems.
We note that, if
$\psi_{n,p}$, $p = 1,\dots,n$, is given by 
$$\psi_{n,p}(X) = \sum_{j - i = p} 
(\nph(N_{i+1}) - \nph(N_i))X(N_{j+1} - N_j), \ \ X\in \cl B(H_1,H_2),$$
then $\psi_n = \sum_{p=1}^n \psi_{n,p}$, 
$\psi_{n,p}\psi_{n,q} = 0$ if $p\neq q$ and $\|\psi_{n,p}\| = 1$, $p = 1,\dots,n$.


\section{Stability under summation with modules of finite width}\label{s_stab}

The main result in this section is Theorem \ref{th_sump}
and the associated Corollary \ref{c_tm}, which show that 
tensor product formulas are preserved under taking weak* closed 
sums with masa-bimodules of finite width. 
En route, we establish a sufficient condition for a
ternary masa-bimodule to possess  $S_{\sigma}$ hereditarily (Theorem \ref{th_trohed}). 
We begin with some lemmas.

\begin{lemma}\label{l_pre}\label{l_inphi}

Let $\phi$ be a weak* continuous completely bounded linear map on 
$\cl B(H_1,H_2)$ and 
$\cl V, \cl V_i\subseteq \cl B(H_1,H_2)$, 
$\cl U, \cl U_i\subseteq \cl B(K_1,K_2)$ be weak* closed subspaces, $i = 1,\dots,n$. 
Suppose that $\cl V$ is invariant under $\phi$.

(i)  \ We have $\tilde{\phi}(\overline{\sum_{i=1}^n \cl V_i \bar\otimes \cl U_i}) 
\subseteq \overline{\sum_{i=1}^n \overline{\phi(\cl V_i)} \bar\otimes \cl U_i}$. 
In particular, $\tilde{\phi}$ leaves $\cl V\bar\otimes \cl U$ invariant.

(ii) 
If $\phi$ is an idempotent then
$$(\ran \phi\bar\otimes\cl U)\cap (\cl V\bar\otimes\cl U)  
= \phi(\cl V)\bar\otimes\cl U 
= \tilde{\phi}(\cl V\bar\otimes\cl U).$$
In particular, ranges of Schur idempotents possess property $S_{\sigma}$.
\end{lemma}

\begin{proof}

(i) Fix $i\in \{1,\dots,n\}$ and
suppose that 
$T\in \cl V_i \bar\otimes \cl U_i$. Then 
$T$ can be approximated in the weak* topology 
by operators of the form $\sum_{j=1}^k A_j\otimes B_j$, 
where $A_j\in \cl V_i$, $B_j\in \cl U_i$, $j = 1,\dots,k$;
therefore,
$\tilde{\phi}(T)$ can be approximated in the weak* topology 
by operators of the form $\sum_{j=1}^k \phi(A_j)\otimes B_j$, 
where $A_j\in \cl V_i$, $B_j\in \cl U_i$, $j = 1,\dots,k$.
Hence, $\tilde{\phi}(T)\in \overline{\phi(\cl V_i)}\bar\otimes \cl U_i$. 
The conclusion now follows from the linearity and the weak* continuity of $\tilde{\phi}$.

(ii) Since $\phi$ is an idempotent, $\phi(\cl V)$ is weak* closed. 
By (i), 
$\tilde{\phi}(\cl V\bar\otimes\cl U)
\subseteq \phi(\cl V)\bar\otimes\cl U$
while, since $\phi(\cl V)\subseteq \cl V$, we 
have $\phi(\cl V)\bar\otimes\cl U \subseteq 
(\ran \phi\bar\otimes\cl U)\cap (\cl V\bar\otimes\cl U)$.
Suppose that $T\in (\ran \phi\bar\otimes\cl U)\cap (\cl V\bar\otimes\cl U)$. 
Then, by (i), $\tilde{\phi}^{\perp}(T) = 0$ and hence 
$$T = \tilde{\phi}(T) \in \tilde{\phi}(\cl V\bar\otimes\cl U).$$

Finally, if 
$$T\in (\ran\phi \bar\otimes\cl B(K_1,K_2))\cap (\cl B(H_1,H_2)\bar\otimes \cl U)$$
then, by the previous paragraph,
$T = \tilde{\phi}(T) \in \phi(\cl B(H_1,H_2))\bar\otimes\cl U$; thus, 
$\ran\phi$ possesses $S_{\sigma}$. 
\end{proof}

\begin{lemma}\label{l_o2}
Let $\phi$ be a weak* continuous completely bounded idempotent acting on $\cl B(H_1,H_2)$, 
$\cl V\subseteq \cl B(H_1,H_2)$ be a weak* closed subspace invariant under $\phi$,
$\cl U\subseteq \cl B(K_1,K_2)$ be a weak* closed subspace and
$\cl W\subseteq \cl B(H_1\otimes K_1,H_2\otimes K_2)$ be a weak* closed
subspace invariant under $\tilde{\phi}$. 
Then 
$$\overline{\cl V \bar\otimes \cl U + \cl W} \cap
\overline{\ran\phi\bar\otimes \cl U + \cl W} = 
\overline{(\ran \phi \cap \cl V) \bar\otimes \cl U + \cl W}.$$
\end{lemma}

\begin{proof}
Suppose 
$$T\in 
\overline{\cl V \bar\otimes \cl U + \cl W} \cap
\overline{\ran\phi\bar\otimes \cl U + \cl W}.$$
By Lemma \ref{l_inphi} and the invariance of $\cl W$ under $\tilde{\phi}$, we have that
$\tilde{\phi}^{\perp}(T) \in \cl W$;
similarly, 
$$\tilde{\phi}(T) \in \overline{\phi (\cl V) \bar\otimes \cl U + \cl W}.$$
It follows that 
$$T = \tilde{\phi}^{\perp}(T) + \tilde{\phi}(T) \in 
\overline{(\ran\phi \cap \cl V) \bar\otimes \cl U + \cl W}.$$
The converse inclusion is trivial. 
\end{proof}

\begin{lemma}\label{l_invf}
Let $\cl V\subseteq \cl B(H_1,H_2)$ (resp. $\cl U\subseteq \cl B(K_1,K_2)$) be 
a weak* closed subspace and $\phi$ (resp. $\psi$) be a 
weak* continuous completely bounded map on $\cl B(H_1,H_2)$ (resp. $\cl B(K_1,K_2)$). Then
$$(\phi\otimes\psi)(\cl F(\cl V,\cl U))\subseteq \cl F(\overline{\phi(\cl V)},\overline{\psi(\cl U)}).$$
Moreover, if $\phi$ and $\psi$ are 
idempotents that leave $\cl V$ and $\cl U$, respectively, invariant, then 
$$(\phi\otimes\psi)(\cl F(\cl V,\cl U)) = \cl F(\phi(\cl V),\psi(\cl U)).$$
\end{lemma}
\begin{proof}
By Lemma \ref{l_pre},
\begin{eqnarray*}
(\phi\otimes\psi)(\cl V\bar\otimes\cl B(K_1,K_2)) & = & 
\tilde{\phi} \circ (\id\otimes\psi)(\cl V\bar\otimes\cl B(K_1,K_2))\\
& \subseteq & \tilde{\phi} (\cl V\bar\otimes\cl B(K_1,K_2))\subseteq 
\overline{\phi(\cl V)}\bar\otimes\cl B(K_1,K_2);
\end{eqnarray*}
similarly, 
$$(\phi\otimes\psi)(\cl B(H_1,H_2)\bar\otimes\cl U)\subseteq 
\cl B(H_1,H_2)\bar\otimes \overline{\psi(\cl U)}.$$
Hence,
\begin{eqnarray*}
(\phi\otimes\psi)(\cl F(\cl V,\cl U)) & = &
(\phi\otimes\psi)((\cl V\bar\otimes\cl B(K_1,K_2)) \cap (\cl B(H_1,H_2)\bar\otimes\cl U))\\
& \subseteq &
(\overline{\phi(\cl V)}\bar\otimes\cl B(K_1,K_2)) \cap (\cl B(H_1,H_2)\bar\otimes\overline{\psi(\cl U)})
= 
\cl F(\overline{\phi(\cl V)},\overline{\psi(\cl U)}).
\end{eqnarray*}

Now suppose that 
$\phi$ and $\psi$ are idempotents that leave $\cl V$ and $\cl U$, respectively, invariant.
Then, clearly, 
$\cl F(\phi(\cl V),\psi(\cl U))\subseteq \cl F(\cl V,\cl U)$. 
On the other hand, if 
$$T\in \cl F(\phi(\cl V),\psi(\cl U)) = (\phi(\cl U)\bar\otimes \cl B(K_1,K_2)) \cap (\cl B(H_1,H_2)\bar\otimes \psi(\cl V))$$
then 
$$\phi\otimes \psi(T) = (\id\otimes\psi)(\phi\otimes \id)(T) = T,$$
and hence $T = \phi\otimes\psi(T)\in (\phi\otimes\psi)(\cl F(\cl V,\cl U)).$
Thus, 
$\cl F(\phi(\cl V),\psi(\cl U))\subseteq \phi\otimes\psi)(\cl F(\cl V,\cl U))$; 
the converse inclusion follows from the previous paragraph. 
\end{proof}

Suppose that $\cl H_k$, $k\in \bb{N}$, are Hilbert spaces and let $\cl H = \oplus_{k\in \bb{N}}\cl H_k$. 
Then every operator $T\in \cl B(\cl H)$ has an operator matrix representation $T = (T_{i,j})$, where 
$T_{i,j}\in \cl B(\cl H_j,\cl H_i)$. Given a family $\cl X = (\cl X_{i,j})_{i,j\in \bb{N}}$, where 
$\cl X_{i,j}\subseteq \cl B(\cl H_j,\cl H_i)$ is a weak* closed subspace, we let
$$\widetilde{\cl X} = \{T = (T_{i,j})\in \cl B(\cl H) : T_{i,j}\in \cl X_{i,j}, i,j\in \bb{N}\}.$$
It can be readily verified that $\widetilde{\cl X}$ is weak* closed. Moreover, 
it follows directly from its definition that
if $\cl Y = (\cl Y_{i,j})_{i,j\in \bb{N}}$ is another such family, 
$\cl Z_{i,j} = \cl X_{i,j}\cap \cl Y_{i,j}$ and $\cl Z = (\cl Z_{i,j})_{i,j\in \bb{N}}$, then 
\begin{equation}\label{eq_inter}
\widetilde{\cl Z} = \widetilde{\cl X}\cap \widetilde{\cl Y}. 
\end{equation}


\begin{lemma}\label{l_dstp}
Let $H_1^k$, $H_2^k$, $k \in \bb{N}$, be Hilbert spaces, 
$H_1 =  \oplus_{k\in \bb{N}} H_1^k$, 
$H_2 =  \oplus_{k\in \bb{N}} H_2^k$, and 
$\cl U \subseteq \cl B(K_1,K_2)$ be a weak* closed subspace. 

(i) If $\cl V_{i,j}\subseteq \cl B(H_1^j,H_2^i)$ is a weak* closed subspace, $i,j \in \bb{N}$, 
$\cl V = (\cl V_{i,j})_{i,j\in \bb{N}}$ and 
$\cl P = (\cl F(\cl V_{i,j},\cl U))_{i,j\in \bb{N}}$, then 
$\cl F(\widetilde{\cl V},\cl U) = \widetilde{\cl P}$.

(ii) \ If $\cl V_k$ is a weak* closed subspace of $\cl B(H_1^k,H_2^k)$, $k\in \bb{N}$, 
then $\cl F(\oplus_{k\in \bb{N}} \cl V_k, \cl U) = \oplus_{k\in \bb{N}} \cl F(\cl V_k, \cl U) $.\end{lemma}
\begin{proof}
(i)
Let $\epsilon_{i,j}$ be the evaluation at the $(i,j)$-entry of a matrix $(X_{l,m})_{l,m\in \bb{N}}$.
We identify $H_i\otimes K_i$ with $\oplus_{k=1}^{\infty} (H_i^k\otimes K_i)$, $i = 1,2$. 
We claim that, for every $\tau \in \cl B(K_1,K_2)_*$, we have 
\begin{equation}\label{eq_lr}
L_{\tau}(\epsilon_{i,j}(T)) = \epsilon_{i,j}(L_{\tau}(T)), \ \ \ T\in \cl B(H_1\otimes K_1,H_2\otimes K_2).
\end{equation}
Indeed, if $T = A\otimes B$, where $A = (A_{i,j})_{i,j\in \bb{N}} \in \cl B(H_1,H_2)$ and $B\in \cl B(K_1,K_2)$, 
then 
$$L_{\tau}(\epsilon_{i,j}(T)) = \tau(B)A_{i,j} = \epsilon_{i,j}(\tau(B)A) = \epsilon_{i,j}(L_{\tau}(T)),$$ and the 
general case follows by linearity and weak* continuity. 
A similar argument shows that, for every 
$\omega\in \cl B(H_1^j,H_2^i)_*$, we have 
\begin{equation}\label{eq_rr}
R_{\omega}(\epsilon_{i,j}(T)) = R_{\omega \circ \epsilon_{i,j}}(T), \ \ \ T\in \cl B(H_1\otimes K_1,H_2\otimes K_2).
\end{equation}

Now suppose that $T = (T_{i,j})\in \cl F(\widetilde{\cl V},\cl U)$. 
Then (\ref{eq_lr}) shows that $L_{\tau}(T_{i,j}) \in \cl V_{i,j}$
for every $\tau \in \cl B(K_1,K_2)_*$, while 
(\ref{eq_rr}) shows that $R_{\omega}(T_{i,j}) \in \cl U$
for every $\omega \in \cl B(H_1^j,H_2^i)_*$, $i,j\in \bb{N}$.
It follows that 
$T_{i,j} \in \cl F(\cl V_{i,j},\cl U)$, $i,j\in \bb{N}$; 
thus, $\cl F(\widetilde{\cl V},\cl U)\subseteq \widetilde{\cl P}$.

Conversely, suppose that $T\in \widetilde{\cl P}$. 
Then
$T_{i,j}\in \cl F(\cl V_{i,j}, \cl U)$ for every $i$ and $j$.
Given $\tau\in \cl B(K_1,K_2)_*$, (\ref{eq_lr}) implies that $L_{\tau}(T)\in \widetilde{\cl V}$.
On the other hand, letting $(E_{i,j})_{i,j}$ be the standard matrix unit system, 
(\ref{eq_rr}) shows that, if $\omega\in \cl B(H_1,H_2)_*$ and 
$\omega_{i,j}\in \cl B(H_1^j,H_2^i)_*$ is given by
$\omega_{i,j}(A) = \omega(A\otimes E_{i,j})$, 
then
$$R_{\omega}(T) 
= \lim_{N\to\infty} \sum_{i,j = 1}^N R_{\omega}(\epsilon_{i,j}(T)\otimes E_{i,j})
= \lim_{N\to\infty} \sum_{i,j = 1}^N R_{\omega_{i,j}}(\epsilon_{i,j}(T)) \in \cl U;$$
thus, $T\in \cl F(\widetilde{\cl V},\cl U)$.

(ii) is a special case of (i) obtained by letting $\cl V_{i,j} = \{0\}$ if $i\neq j$. 
\end{proof}

\begin{corollary}\label{c_psmp}
In the notation of Lemma \ref{l_dstp}, if $\cl V_{i,j}$ and $\cl V_k$ 
have property $S_{\sigma}$ for 
every $i,j$ and $k$, 
then $\widetilde{\cl V}$ and $\oplus_{k\in \bb{N}} \cl V_k$ do so as well. 
\end{corollary}

We note that the conclusion regarding the direct sum in the 
last corollary also follows from 
\cite[Proposition 1.11]{kraus_tams}.

It was shown in \cite{kraus_tams} that every von Neumann algebra with abelian commutant 
possesses property $S_{\sigma}$. We will shortly show that the same holds for ternary masa-bimodules.
We first need a lemma.

\begin{lemma}\label{l_co}
Assume that $H_1 = \ell^2$, $(e_i)_{i\in \bb{N}}$
is its standard orthonormal basis, and $H_2$ is a Hilbert space. 
Let $P_i$ be the rank one projection whose 
range is spanned by $e_i$, and
$Q_i$ be a projection on $H_2$, $i\in \bb{N}$.
The space
$$\cl V = \{T\in \cl B(H_1,H_2) : TP_i = Q_iTP_i, i\in \bb{N}\}$$ 
possesses property $S_{\sigma}$.
\end{lemma}

\begin{proof}
Let $\cl U\subseteq \cl B(K_1,K_2)$ be a weak* closed subspace
and
$T\in \cl F(\cl V,\cl U)$. Then $T = \sum_{i=1}^{\infty} T(P_i\otimes I)$
(where the series converges in the weak* topology); it hence suffices 
to show that $T(P_i\otimes I)\in \cl V\bar\otimes \cl U$ for each $i$. 
However, since $T\in \cl V\bar\otimes \cl B(K_1,K_2)$, we 
have that $T(P_i\otimes I) = (Q_i\otimes I)T(P_i\otimes I)$. 
But $(Q_i\otimes I)T(P_i\otimes I) \in (Q_i\cl B(H_1,H_2)P_i)\bar\otimes\cl U$, 
and the latter space is contained in $\cl V\bar\otimes\cl U$. 
It follows that $T\in \cl V\bar\otimes\cl U$. 
\end{proof}

\begin{theorem}\label{th_trohed}
Let $\cl D_1\subseteq \cl B(H_1)$ and $\cl D_2\subseteq \cl B(H_2)$ be 
masas and $\cl M\subseteq \cl B(H_1,H_2)$ be a ternary $\cl D_2,\cl D_1$-bimodule. 
Then $\cl M$ possesses property $S_{\sigma}$.
If, moreover, $\cl M$ does
not contain subspaces of the form $\cl B(EH_1,FK_1)$, where 
$E\in \cl D_1$ and $F\in \cl D_2$ are non-zero non-atomic projections, then 
every masa-bimodule $\cl V$ with $\cl V\subseteq \cl M$ possesses property $S_{\sigma}$.
\end{theorem}
\begin{proof}
We have (see, {\it e.g.}, \cite{houston}) 
that, up to unitary equivalence,  
$$\cl M = \left(\oplus_{j=1}^m M_{l_j,k_j}(\cl D)\right)\oplus 
\left(\oplus_{j=1}^{l} \cl B(E_kH_1,F_kH_2)\right),$$
where $m,l\in \bb{N}\cup\{\infty\}$, 
$\cl D$ is the multiplication masa of $L^{\infty}(0,1)$ 
acting on $L^2(0,1)$
and $E_k$ (resp. $F_k$) is a projection in $\cl D_1$ (resp. $\cl D_2$).
Since $\cl D$ possesses $S_{\sigma}$ \cite{kraus_tams}, Corollary \ref{c_psmp} 
implies that $\cl M$ does so as well. 

Suppose that $\cl M$ does
not contain subspaces of the form $\cl B(EH_1,FK_1)$, where 
$E\in \cl D_1$ and $F\in \cl D_2$ are non-zero non-atomic projections.
Then,  for each $k$, either $E_k$ is totally atomic or $F_k$ is such.
Let $\cl V\subseteq \cl M$ be a masa-bimodule.
Then $\cl V  = \cl M_0 \oplus \cl W_1 \oplus \cl W_2^*$, where 
$\cl M_0 \subseteq \oplus_{j=1}^m M_{l_j,k_j}(\cl D)$, while
$\cl W_1$ and $\cl W_2$ have the form described in Lemma \ref{l_co}. 
We have that $\cl M_0 = \oplus_{j=1}^m \cl U_j,$ where
$\cl U_j = \{(A_{p,q}^j) : A_{p,q}^j\in \cl D_{p,q}^j\}$, and
$\cl D_{p,q}^j \subseteq \cl D$, for all $p = 1,\dots,l_j$ and $q = 1,\dots,k_j$. 
We have that
$\cl D_{p,q}^j$ is itself a (continuous) masa and hence possesses property $S_{\sigma}$
\cite{kraus_tams}. 
It follows from Corollary \ref{c_psmp} that $\cl M_0$ has $S_{\sigma}$. 
On the other hand, every masa-bimodule contained in $\cl W_1$ or $\cl W_2^*$
has the form described in Lemma \ref{l_co}. 
By Corollary \ref{c_psmp} and Lemma \ref{l_co},
$\cl V$ has property $S_{\sigma}$.
\end{proof}

\begin{lemma}\label{l_inmm}
Let $\cl M,\cl V\subseteq \cl B(H_1,H_2)$ be masa-bimodules with $\cl M$ ternary, and 
let $\cl U\subseteq \cl B(K_1,K_2)$ be a weak* closed subspace.
Then
$$(\cl M\bar\otimes\cl U) \cap (\cl V\bar\otimes\cl U) = (\cl M\cap \cl V)\bar\otimes\cl U.$$
\end{lemma}
\begin{proof}
Up to unitary equivalence,
$$\cl M = (\oplus_{j=1}^m M_{l_j,k_j}(\cl D)) \oplus \left(\oplus_{i=1}^k \cl B(E_i H_1,F_i H_2)\right),$$
where $k,m\in \bb{N}\cup\{\infty\}$, 
$(E_i)_{i=1}^k$ and $(F_i)_{i=1}^k$ are families of mutually orthogonal
projections belonging to the corresponding masas, while 
$\cl D$ is a continuous masa.
Let $\cl M_a = \oplus_{i=1}^k \cl B(E_i H_1,F_i H_2)$ be the atomic part of $\cl M$
and $\cl M_c =  \oplus_{j=1}^m M_{l_j,k_j}(\cl D_j)$ be its continuous part, both naturally identified with
subspaces of $\cl M$. We have a natural identification
$$\cl M\bar\otimes\cl U = (\cl M_a\bar\otimes\cl U) + (\cl M_c\bar\otimes\cl U)$$ and
$$(\cl M\bar\otimes\cl U) \cap (\cl V\bar\otimes\cl U) =
((\cl M_a\bar\otimes\cl U)\cap (\cl V\bar\otimes\cl U)) + ((\cl M_c\bar\otimes\cl U)\cap (\cl V\bar\otimes\cl U))$$
(for the second identity, we use the fact that, if $\theta$ is the Schur idempotent 
given by $\theta(X) = QXP$, where $P = \vee_i E_i$ and $Q = \vee_i F_i$, then 
$\tilde{\theta}$ leaves $\cl V\bar\otimes\cl U$ invariant and maps $\cl M\bar\otimes\cl U$
onto $\cl M_a\bar\otimes\cl U$).

Let $\phi$ be the map on $\cl B(H_1,H_2)$ given by 
$\phi(X) = \sum_{i=1}^{\infty} F_i X E_i$. Then $\phi$ is a 
contractive Schur idempotent and so,
by Lemma \ref{l_inphi}, we have
\begin{eqnarray*}
(\cl M_a\bar\otimes\cl U)\cap (\cl V\bar\otimes\cl U)
& = & (\ran\phi \bar\otimes\cl U)\cap (\cl V\bar\otimes\cl U)
=  \phi(\cl V)\bar\otimes \cl U\\
& = & (\cl M_a \cap \cl V) \bar\otimes \cl U \subseteq 
(\cl M \cap \cl V) \bar\otimes \cl U.
\end{eqnarray*}

Suppose that
$T\in (\cl M_c\bar\otimes\cl U)\cap (\cl V\bar\otimes\cl U)$. Then $L_{\tau}(T)\in \cl M_c\cap \cl V$
and $R_{\omega}(T)\in \cl U$ for all $\tau\in \cl B(K_1,K_2)_*$ and 
all $\omega\in \cl B(H_1,H_2)_*$.
Thus, 
$$T\in \cl F(\cl M_c\cap \cl V,\cl U) = (\cl M_c\cap \cl V)\bar\otimes\cl U \subseteq (\cl M\cap \cl V)\bar\otimes\cl U,$$
where the equality follows from Theorem \ref{th_trohed}, applied to the 
masa-bimdule $\cl M_c$. 
We hence showed that
$$(\cl M\bar\otimes\cl U) \cap (\cl V\bar\otimes\cl U) \subseteq (\cl M\cap \cl V)\bar\otimes\cl U;$$
the converse inclusion is trivial.
\end{proof}

\begin{lemma}\label{l_intem}
Let $(\cl M_n)_{n\in \bb{N}}$ be a descending sequence of 
ternary masa-bimodules in $\cl B(H_1,H_2)$ 
and $\cl M = \cap_{n\in \bb{N}}\cl M_n$. 
If $\cl U\subseteq \cl B(K_1,K_2)$ is a weak* closed subspace then 
$\cap_{n\in \bb{N}} (\cl M_n\bar\otimes\cl U) = \cl M\bar\otimes\cl U$. 
\end{lemma}
\begin{proof}
The inclusion $\cl M\bar\otimes\cl U \subseteq \cap_{n\in \bb{N}} (\cl M_n\bar\otimes\cl U)$ is  trivial. 
Suppose that $T\in \cl M_n\bar\otimes\cl U$ for each $n$. Then 
$L_{\tau}(T)\in \cl M_n$ for all $n$, and so $L_{\tau}(T)\in \cl M$, for all $\tau\in \cl B(K_1,K_2)_*$. 
On the other hand, $R_{\omega}(T)\in \cl U$ for all $\omega\in \cl B(H_1,H_2)_*$. It follows that 
$T\in \cl F(\cl M,\cl U)$ and since $\cl M$ 
possesses property $S_{\sigma}$ (Theorem \ref{th_trohed}), 
we conclude that $T\in \cl M\bar\otimes\cl U$. 
\end{proof}

We are now ready to prove the main result of this section.

\begin{theorem}\label{th_sump}
Let $\cl V\subseteq\cl B(H_1,H_2)$ be a masa-bimodule,
$\cl B\subseteq\cl B(H_1,H_2)$ be a masa-bimodule of finite width and 
$\cl U\subseteq \cl B(K_1,K_2)$ be a weak* closed subspace. 
If $\cl F(\cl V,\cl U) = \cl V\bar\otimes\cl U$ then 
$\cl F(\overline{\cl V + \cl B},\cl U) = \overline{\cl V+\cl B}\bar\otimes\cl U$.
In particular, if $\cl V$ possesses property $S_{\sigma}$ then
$\overline{\cl V + \cl B}$ does so as well. 
\end{theorem}
\begin{proof}
We use induction on the length $k$ of $\cl B$. If $k = 0$, that is, $\cl B = \{0\}$,
the statement is trivial. Suppose that it holds for masa-bimodules 
of length at most $k$, let $\cl B$ be a masa-bimodule of length $k$, and
let $\cl A$ be a nest algebra bimodule.
Let $(\phi_n,\psi_n,\cl M_n,\cl W_n)_{n\in \bb{N}}$ be a decomposition scheme for $\cl A$, and set
$\theta_n = \id - (\phi_n + \psi_n )$.

Suppose 
$T\in \cl F(\overline{\cl V + (\cl B\cap \cl A)}, \cl U)$ and 
write $T = \tilde{\phi}_n(T) + \tilde{\psi}_n(T) + \tilde{\theta}_n(T).$
Since $\cl A \subseteq \ran(\phi_n + \psi_n)$, we have that $\theta_n(\cl A) = \{0\}$ and 
hence, by Lemma \ref{l_invf}, 
$$\tilde{\theta}_n(T)\in \cl F(\cl V,\cl U) = \cl V \bar\otimes\cl U.$$
By Lemma \ref{l_invf} and the fact that 
$\psi_n(\cl B) = \ran\psi_n \cap \cl B$, we have 
\begin{eqnarray*}
\tilde{\psi}_n(T)
& \in & \cl F(\psi_n(\overline{\cl V + (\cl B\cap \cl A)}), \cl U)
= \cl F(\overline{\psi_n(\cl V) + \psi_n(\cl B\cap \cl A)}, \cl U)\\
& = & \cl F(\overline{\psi_n(\cl V) + \psi_n(\cl B)}, \cl U)
= \tilde{\psi}_n(\cl F(\overline{\cl V + \cl B}, \cl U))
= \tilde{\psi}_n(\overline{\cl V + \cl B}\bar\otimes \cl U)\\
& = & \psi_n(\overline{\cl V + \cl B})\bar\otimes \cl U
\subseteq \overline{\cl V + \psi_n(\cl B)}\bar\otimes \cl U)
\subseteq \overline{\cl V + (\cl B\cap\cl A)}\bar\otimes \cl U.
\end{eqnarray*}

On  the other hand, by Lemmas \ref{l_invf} and \ref{l_inmm} 
and the facts that $\phi_n(\cl V) = \cl V\cap \cl M_n$
and $\phi_n(\cl B) = \cl B\cap \cl M_n$, we have

\begin{eqnarray*}
\tilde{\phi}_n(T) & \in & \cl F(\overline{\phi_n(\cl V)  + (\phi_n( \cl B) \cap \cl M_n)}, \cl U)\\
& = &
(\overline{\cl V\cap \cl M_n + \cl B\cap \cl M_n}\bar\otimes\cl B(K_1,K_2))\cap
(\cl B(H_1,H_2)\bar\otimes\cl U)\\
& = &
(\overline{\cl V + \cl B}\cap \cl M_n\bar\otimes\cl B(K_1,K_2))\cap
(\cl B(H_1,H_2)\bar\otimes\cl U)\\
& = &
(\overline{\cl V + \cl B}\bar\otimes\cl B(K_1,K_2))\cap
(\cl M_n\bar\otimes\cl B(K_1,K_2))\cap
(\cl B(H_1,H_2)\bar\otimes\cl U).
\end{eqnarray*}

Let $S$ be a weak* cluster point of $(\tilde{\phi}_n(T))_{n\in \bb{N}}$ and set
$\cl M = \cap_{n\in \bb{N}}\cl M_n$. 
Using Lemmas \ref{l_inmm}, \ref{l_intem} and the inductive assumption, we have 
\begin{eqnarray*}
S & \in & (\overline{\cl V + \cl B}\bar\otimes\cl B(K_1,K_2))\cap
(\cl M\bar\otimes\cl B(H_1,H_2))\cap (\cl B(H_1,H_2)\bar\otimes\cl U)\\
& = & \cl F(\overline{\cl V + \cl B},\cl U) \cap \cl F(\cl M,\cl U)
=  (\cl M\bar\otimes\cl U)\cap (\overline{\cl V + \cl B}\bar\otimes \cl U)
=  (\cl M\cap\overline{\cl V + \cl B})\bar\otimes\cl U.
\end{eqnarray*}
Every ternary masa-bimodule has finite width (in fact, it is the
intersection of two nest algebra bimodules \cite{kt}) and 
hence, by \cite[Theorem 2.10]{eletod},
we have that
$$\cl M\cap\overline{\cl V + \cl B} \subseteq \overline{\cl V + \cl M}\cap\overline{\cl V + \cl B} =
\overline{\cl V + (\cl B\cap \cl M)} \subseteq \overline{\cl V + (\cl B\cap \cl A)}.$$
It follows that
$$S\in \overline{\cl V + (\cl B\cap \cl A)}\bar\otimes\cl U.$$
On the other hand, by the second paragraph of the proof,
$$T - S = \lim_{n\rightarrow\infty} (\tilde{\psi}_n(T) + \tilde{\theta}_n(T))
\in \overline{\cl V + (\cl B\cap \cl A)}\bar\otimes\cl U.$$
Hence,
$$T = S + (T - S) \in \overline{\cl V + (\cl B\cap \cl A)}\bar\otimes\cl U,$$ and the proof is complete.
\end{proof}

The next corollary is an immediate consequence of 
Theorem \ref{th_sump}. 
It extends the fact, established in \cite{hok}, that 
CSL algebras of finite width possess property $S_{\sigma}$.

\begin{corollary}\label{c_tm}
If $\cl B_i$, $i = 1,\dots,n$, are masa-bimodules of finite width, then
$\overline{\cl B_1 + \cdots + \cl B_n}$ has property $S_{\sigma}$.
\end{corollary}


\section{Intersections and spans}\label{s_insum}

In this section, we establish an intersection formula for weak* closures of 
spans of subspaces of the form $\cl B\bar\otimes\cl U$, where 
$\cl B$ is a masa-bimodule of finite width (see Theorem \ref{epsilon} and
Corollary \ref{epsilona}). This result will be used in Section \ref{s_appos} to 
study questions about operator synthesis. 

We fix masas $\cl D_1\subseteq \cl B(H_1)$ and 
$\cl D_2\subseteq \cl B(H_2)$. 
All Schur idempotents we consider are relative to the pair
$(\cl D_1,\cl D_2)$ and act on $\cl B(H_1,H_2)$.
We will say that a sequence $(\psi_n)_{n\in \bb{N}}$ 
of Schur idempotents is \emph{nested} if 
$\ran\psi_{n+1}\subseteq \ran\psi_n$ for all $n\in \bb{N}$. 
Before formulating the main result of this section, Theorem \ref{epsilon}, 
we state three propositions which will be needed in its proof. 
We first recall that the ranges of contractive Schur idempotents 
on $\cl B(H_1,H_2)$ are 
ternary masa-bimodules of the form $\oplus_k \cl B(E_k\cl H_1,F_k\cl H_2)$,
where $(E_k)_k\subseteq \cl D_1$ and $(F_k)_k\subseteq \cl D_2$ 
are families of mutually orthogonal
projections (see, {\it e.g.}, \cite{kp}).

\bigskip

\noindent {\bf Notation. } 
For the rest of this section, we let
$\cl B_i\subseteq \cl B(H_1,H_2)$ be a masa-bimodule of finite width,
$\cl U, \cl V, \cl U_i$ be weak* closed subspaces of $\cl B(K_1,K_2)$, $i = 1,\dots,r$,
and $\cl W = \sum_{i=1}^r \cl B_i\bar\otimes \cl U_i$.

\smallskip

\begin{proposition}\label{mazi}
Let $(\psi_i)_{i\in \mathbb{N}}$ be a nested sequence of contractive Schur idempotents.

(i) \ \ Let $(\phi _i)_{i\in \bb{N}}$ be a nested sequence of contractive Schur idempotents. 
Then the subspaces
$$\cap _{k,i}\overline{(\ran \psi_i\cap \ran \phi_k )\bar \otimes \cl U+\ran \phi_k \bar \otimes \cl V + \cl W} \ \mbox{ and } \ $$
$$\overline{((\cap _{i} \ran \psi _i)\cap(\cap_k  \ran \phi _k))\bar \otimes \cl U+(\cap _k\ran \phi_k) \bar \otimes \cl V + \cl W}$$
coincide.

(ii)  \ Let $\cl B$ be a nest algebra bimodule. Then
$$\cap _i \overline{ \ran \psi _i\bar \otimes \cl U + \cl B\bar \otimes \cl V + \cl W} = \overline{(\cap _i\ran \psi _i) \bar \otimes \cl U + \cl B\bar \otimes \cl V + \cl W}.$$

(iii)  Let $\cl B$ be a nest algebra bimodule. Then the subspaces
$$\cap _i \overline{\ran\psi_i\bar \otimes \cl U + (\ran \psi _i\cap \cl B)\bar \otimes \cl V + \cl W} \ \mbox{ and } \ $$
$$\overline{(\cap _i\ran\psi_i) \bar \otimes \cl U+((\cap _i\ran \psi _i)\cap \cl B)\bar \otimes \cl V + \cl W}$$ coincide.

(iv) \ Let $\cl B$ be a masa-bimodule of finite width. Then the subspaces
$$\cap _i \overline{\ran\psi_i\bar \otimes \cl U + (\ran \psi _i\cap \cl B)\bar \otimes \cl V + \cl W} \ \mbox{ and } \ $$
$$\overline{(\cap _i\ran\psi_i) \bar \otimes \cl U+((\cap _i\ran \psi _i)\cap \cl B)\bar \otimes \cl V + \cl W}$$ coincide.

(v) \ \
$\cap _i \overline{\ran\psi_i\bar \otimes \cl U + \cl W} =
\overline{(\cap _i\ran\psi_i) \bar \otimes \cl U + \cl W}.$
\end{proposition}

We note that part (ii) of the previous proposition is more general than (v); 
however, for the purpose of its proof it will be convenient to formulate 
these statements separately.

\begin{proposition}\label{gama}
Let $\cl M $ be a ternary masa-bimodule and
$\cl C $ be a nest algebra bimodule. Then
$$\overline{ \cl M\bar \otimes \cl U + \cl W} \cap \overline{ \cl C\bar \otimes \cl U + \cl W}
= \overline{ (  \cl M\cap \cl C)\bar \otimes \cl U + \cl W}.$$
\end{proposition}

\smallskip

\begin{proposition}\label{delta}
Let $\cl M $ be a ternary masa-bimodule and
$\cl C $ be a masa-bimodule of finite width.
Then
$$\overline{ \cl M\bar \otimes \cl U + \cl W}  \cap \overline{ \cl C\bar \otimes \cl U + \cl W}
= \overline{ (  \cl M\cap \cl C)\bar \otimes \cl U + \cl W}.$$
\end{proposition}

\smallskip

\begin{theorem}\label{epsilon}
Let $\cl C_j\subseteq \cl B(H_1,H_2)$ be a masa-bimodule of finite width,
$j = 1,\dots,m$. Then
$$\cap_{j=1}^m \overline{\cl C_j\bar \otimes \cl U + \cl W} =
\overline{(\cap_{j=1}^m\cl C_j)\bar \otimes \cl U + \cl W}.$$
\end{theorem}

The proof of the above results will be given simultaneously, using induction on
the number $r$ of terms in the sum $\cl W =  \sum_{i=1}^r \cl B_i\bar\otimes \cl U_i$
and will be split into a number of lemmas.
The first series of steps, namely Lemmas \ref{diokyklos}--\ref{lepsilon},
provide the base of the induction. We will refer to the statements in Proposition \ref{mazi}
by their corresponding numbers (i) -- (v).
It will be convenient to assume that $\cl W = \{0\}$ when $r = 0$. 

Given the notation in Proposition \ref{mazi},  
throughout the proofs, we will set for brevity
$$\cl N = \cap_{i} \ran\psi_i \ \ \mbox{ and } \ \ \cl R = \cap_k \ran\phi_k.$$ 
The proofs of the lemmas in this section will all use 
the following idea:
Let $\Omega$ be a weak* closed subspace of operators 
and $(\rho_n)_{n\in \bb{N}}$ be a nested sequence of contractive idempotents
with $\cap_{n=1}^{\infty}\ran\rho_n \subseteq \Omega$.
In order to prove that a certain operator 
$T$ belongs to $\Omega$, it suffices to show that 
$\rho_n^{\perp}(T)\in \Omega$ for each $n\in \bb{N}$.
Indeed, letting $S$ be a weak* cluster point of
the sequence $(\rho_n(T))_{n\in \bb{N}}$, we have that 
$S\in \cap_{n=1}^{\infty}\ran\rho_n \subseteq \Omega$. On the other hand, the 
identities $T = \rho_n(T) + \rho_n^{\perp}(T)$, $n\in \bb{N}$, 
show that $T-S$ is a weak* cluster point of 
the sequence $(\rho_n^{\perp}(T))_{n\in \bb{N}}$, and hence 
it belongs to $\Omega$; therefore, 
$T = S + (T-S) \in \Omega$.

\medskip

\begin{lemma}\label{diokyklos}
Proposition \ref{mazi} (i) holds if $r = 0$.
\end{lemma}
\begin{proof}
Let 
$$\Omega = \overline{(\cl N\cap \cl R)\bar \otimes \cl U + \cl R \bar \otimes \cl V}$$ 
and fix
$$X\in \cap _{k,i} \overline{(\ran \psi _i\cap \ran \phi_k )\bar \otimes \cl U+\ran \phi_k \bar \otimes \cl V}.$$
By Lemma \ref{l_pre}, $\tilde{\psi}_i^\bot (X)\in \ran \phi _k\bar \otimes \cl V$ for all $k,i.$ By Lemma \ref{l_intem},
$\tilde{\psi}_i^\bot (X)\in \cl R\bar \otimes \cl V\subseteq \Omega$, $i\in \bb{N}$.
On the other hand, for all $i\in \bb{N}$, we have, by Lemma \ref{l_pre},
\begin{eqnarray*}
\tilde{\psi_ i}(X)
& \in & 
\overline{ (\ran \psi _i\cap \ran \phi_i) \bar \otimes \cl U
+ (\ran \psi _i\cap \ran \phi_i) \bar \otimes \cl V}\\
& = & 
(\ran \psi _i\cap \ran \phi_i) \bar \otimes (\overline{\cl U+\cl V}).
\end{eqnarray*}
By Lemma \ref{l_intem},
any weak* cluster point $S$ of the sequence 
$(\tilde{\psi}_i (X))_{i\in \bb{N}}$ belongs to 
$(\cl N\cap \cl R)\bar \otimes (\overline{\cl U+\cl V})$,
a subset of $\Omega$. Thus, $X = (X-S) + S\in \Omega$.
\end{proof}

\begin{lemma}\label{triakyklos}
Proposition \ref{mazi} (ii) holds if $r = 0$.
\end{lemma}
\begin{proof}
Let 
$\Omega = \overline{\cl N\bar \otimes \cl U + \cl B\bar \otimes \cl V}$ and fix
$$X\in \cap _i \overline{\ran \psi _i\bar \otimes \cl U+\cl B\bar \otimes \cl V}.$$
Let $(\phi_k,\theta_k,\cl M_k,\cl Z_k)_{k\in \bb{N}}$ be a decomposition scheme for $\cl B$. 
By Lemma \ref{l_pre}, 
$$ \tilde{\phi}_k(X) \in \cap _i \overline{(\ran \psi _i\cap \cl M_k)\bar \otimes \cl U+
\cl M_k\bar \otimes \cl V}.$$
If $S$ is a weak* cluster point of the sequence $(\tilde{\phi}_k(X))_{k\in \bb{N}}$, 
then 
$$S\in \cap _{k, i} \overline{(\ran \psi _i\cap \cl M_k)\bar \otimes \cl U+
\cl M_k\bar \otimes \cl V}.$$  
By Lemma \ref{diokyklos}, $S\in  \Omega .$ It hence suffices to prove
that $ \tilde{\phi}_k^\bot (X) \in \Omega $ for all $k\in \bb{N}$. Observe that
$$\tilde{\phi}_k^\bot (X)\in \cap _i\overline{ \ran \psi _i \bar \otimes \cl U+\cl Z_k\bar \otimes \cl V  }.$$
Using Lemmas \ref{l_pre} and \ref{l_intem}, we see that
$$ \tilde{\theta}_k^\bot (\tilde{\phi}_k^\bot (X))\in  \cap _i(\ran \psi _i\bar \otimes \cl U) =
\cl N \bar \otimes \cl U\subseteq \Omega.$$ 
Write $\theta_k = \sum_{p=1}^k \theta_{k,p}$, where $\theta_{k,p}$ is a 
contractive Schur idempotent whose range is contained in $\cl B$
(see the last paragraph of Section \ref{s_prel}). Then 
$$ \tilde{\theta}_{k,p}(\tilde{\phi}_k^\bot (X))\in
\cap _i\overline{(\ran \psi _i \cap \ran\theta_{k,p} )\bar \otimes \cl U+
\ran\theta_{k,p} \bar \otimes \cl V}.$$ 
By Lemma \ref{diokyklos},
$$\tilde{\theta}_{k,p}(\tilde{\phi}_k^\bot (X))\in \overline{\cl N \bar \otimes 
\cl U + \ran\theta_{k,p} \bar \otimes \cl V}\subseteq \Omega.$$
Hence, $\tilde{\phi}_k^{\bot}(X) = \tilde{\theta}_k^{\bot}(\tilde{\phi}_k^{\bot}(X)) + \sum_{p=1}^k \tilde{\theta}_{k,p}(\tilde{\phi}_k^{\bot}(X))\in \Omega$
for all $k\in \bb{N}$ and the proof is complete.
\end{proof}

\begin{lemma}\label{alphakyklos}
Proposition \ref{mazi} (iii) holds if $r = 0$.
\end{lemma}
\begin{proof}
Let 
$\Omega = \overline{\cl N  \bar \otimes \cl U+(\cl N \cap \cl B)\bar \otimes \cl V}$
and fix
$$X\in \cap _i \overline{\ran \psi _i\bar \otimes \cl U+(\ran \psi _i\cap \cl B)\bar \otimes \cl V}.$$
Let $(\phi_k,\theta_k,\cl M_k,\cl Z_k)_{k\in \bb{N}}$ be a decomposition scheme for $\cl B$.
By Lemma \ref{l_pre}, 
$$\tilde{\phi}_k^\bot (X)\in \cap _i\overline{\ran \psi _i\bar \otimes \cl U+(\ran \psi _i\cap \cl Z_k)
\bar \otimes \cl V }.$$ 
By Lemmas  \ref{l_pre} and  \ref{l_intem},
$$\tilde{ \theta}_k^\bot ( \tilde{\phi}_k^\bot (X) )\in 
\cap _i( \ran \psi _i \bar \otimes \cl U)
= \cl N \bar \otimes \cl U \subseteq \Omega,$$ 
and by Lemmas \ref{l_pre}, \ref{l_o2} and \ref{l_intem},
\begin{eqnarray*}
\tilde{\theta}_k (\tilde{\phi}_k^\bot (X)) & \in  & 
\cap_i\left((\ran \psi _i \cap \cl Z_k)\bar \otimes \overline{\cl U+\cl V}\right)\\
& \subseteq & \left(\cap_i\left(\ran \psi _i\bar \otimes \overline{\cl U+\cl V}\right)\right)\cap
\left(\cl Z_k \bar \otimes \overline{\cl U +\cl V}\right)\\
& = & 
(\cl N\bar \otimes \overline{\cl U+\cl V}) \cap (\cl Z_k\bar \otimes \overline{\cl U+\cl V})
= (\cl N\cap \cl Z_k)\bar \otimes \overline{\cl U+\cl V}\subseteq \Omega; 
\end{eqnarray*}
thus, 
$ \tilde{\phi}_k^\bot (X) \in \Omega.$ 
On the other hand, by Lemma \ref{l_pre},
$$\tilde{\phi}_k (X) \in (\ran \psi _k\cap \ran \phi _k)\bar \otimes \overline{\cl U+\cl V}, 
\ \ \ k\in \bb{N}.$$ 
Therefore, if $S$ is a weak* cluster point of the sequence 
$(\tilde{\phi}_k (X))_{k\in \bb{N}}$,
then, by Lemma \ref{l_intem}, 
$$S  \in (\cl N \cap \cl R)\bar \otimes \overline{\cl U+\cl V}\subseteq \Omega .$$
The proof is complete.
\end{proof}

\begin{lemma}\label{vitakyklos}
Proposition \ref{mazi} (iv) holds if $r = 0$.
\end{lemma}
\begin{proof}
Let 
$$\Omega = \overline{\cl N  \bar \otimes \cl U+(\cl N \cap \cl B)\bar \otimes \cl V}.$$ 
We use induction on the width $n$ of $\cl B.$ 
If $n=1$, the conclusion follows from Lemma \ref{alphakyklos}.
Suppose that the statement holds for masa-bimodules of width at most $n-1$ and let 
$\cl B=\cap _{l=1}^n\cl A_l$, where $\cl A_l$ is a nest algebra bimodule, $l = 1,\dots,n$. 
Fix
$$X\in  \cap _i \overline{ \ran \psi _i\bar \otimes\cl U + (\ran \psi _i\cap \cl B)\bar \otimes \cl V }.$$
By the inductive assumption, $X$ belongs to both
$$\overline{\cl N \bar \otimes \cl U
+ (\cl N \cap (\cap _{l=1}^{n-1}\cl A_l))
\bar \otimes \cl V}$$
and 
$$ \overline{\cl N \bar \otimes \cl U+(\cl N \cap \cl A_n)\bar \otimes \cl V}.$$
Let $(\phi_k,\theta_k,\cl M_k,\cl Z_k)_{k\in \bb{N}}$ be a decomposition scheme for $\cl A_n$.
For a fixed $k$, we have that $ \tilde{\phi}_k^\bot (X) $ belongs to the intersection of the spaces
$$\overline{\cl N \bar \otimes \cl U+(\cl N \cap (\cap _{l=1}^{n-1}\cl A_l))\bar \otimes \cl V} \ \ 
\mbox{ and } \ \ 
\overline{\cl N \bar \otimes \cl U+(\cl N \cap \cl Z_k) \bar \otimes \cl V}.$$
By Lemma \ref{l_pre},
$$\tilde{\theta}_k^\bot (\tilde{\phi}_k^\bot (X))\in \cl N \bar \otimes \cl U\subseteq \Omega$$ and
$$\tilde{\theta}_k( \tilde{\phi}_k^\bot (X) )\in \overline{ \cl N  \bar \otimes \cl U
+ (\cl N \cap (\cap _{i=1}^{n-1}\cl A_l)\cap \cl Z_k)\bar \otimes \cl U}
\subseteq \Omega, $$
since $\cl Z_k\subseteq \cl A_n$.
Thus, 
$$\tilde{\phi}_k^\bot (X) = 
\tilde{\theta}_k(\tilde{\phi}_k^\bot (X)) + \tilde{\theta}_k^{\perp}(\tilde{\phi}_k^\bot (X))
 \in \Omega .$$
Let $S$ be a weak* cluster point of $(\tilde{\phi}_k(X))_{k\in \bb{N}}$.
Then 
$$S\in \cap _{k}\overline{((\ran \psi _k\cap \ran \phi _k ) \bar \otimes \cl U+(\ran \psi _k\cap \ran \phi _k \cap (\cap _{l=1}^{n-1}\cl A_l))
\bar \otimes
\cl V )}.$$ By the inductive assumption, the latter space coincides with
$$\overline{(\cl N \cap \cl R) \bar \otimes \cl U
+ (\cl N \cap \cl R) \cap (\cap _{l=1}^{n-1}\cl A_l))\bar \otimes\cl V )},$$ 
which is a subset of $\Omega$ since $\cl R  \subseteq \cl A_n$. 
\end{proof}

\begin{lemma}\label{4kyklos}
Proposition \ref{mazi} (v) holds if $r = 1$.
\end{lemma}

\begin{proof}
Let 
$\Omega = \overline{\cl N  \bar \otimes \cl U + \cl W}.$
We set $\cl B = \cl B_1$ and $\cl V = \cl U_1$, and use induction on the width $n$ of $\cl B$.
For $n=1$ the conclusion follows from Lemma \ref{triakyklos}. Suppose that 
the statement holds if the length of $\cl B$ does not exceed $n-1$ and assume that 
$\cl B=\cap _{l=1}^n\cl A_l$
where $\cl A_l$  is a nest algebra bimodule, $l = 1,\dots,n$. 
Let $(\phi_k,\theta_k,\cl M_k,\cl Z_k)_{k\in \bb{N}}$ be a decomposition scheme for $\cl A_n$.
Fix 
$$X\in \cap _i \overline{\ran \psi _i\bar \otimes \cl U+\cl B\bar \otimes \cl V}.$$
By the inductive assumption, $X$ belongs to the intersection of the spaces
 $$\overline{\cl N \bar \otimes \cl U+( \cap _{l=1}^{n-1}\cl A_l )\bar \otimes \cl V}
\ \mbox{ and } \ \overline{\cl N \bar \otimes \cl U + \cl A_n\bar \otimes \cl V}.$$
By Lemma \ref{l_pre}, for a fixed $k$, 
$\tilde{\phi}_k^\bot (X) \in \overline{\cl N \bar \otimes \cl U+\cl Z_k\bar \otimes \cl V}$.
Therefore,
$\tilde{\theta}_k^\bot (\tilde{\phi}_k^\bot (X))\in \cl N \bar \otimes \cl U$,
and hence $\tilde{\theta}_k^\bot (\tilde{\phi}_k^\bot (X)) \in \Omega$.
On the other hand, 
$$\tilde{\theta}_k(\tilde{\phi}_k^\bot (X))\in 
\overline{ \cl N \bar \otimes \cl U +((\cap _{l=1}^{n-1}\cl A_l)\cap \cl Z_k)
\bar \otimes \cl V}\subseteq \Omega $$ 
since $\cl Z_k\subseteq \cl A_n$. 
Thus, $\tilde{\phi}_k^\bot (X) \in \Omega$, for all $k\in \bb{N}$. 
It hence suffices to prove that, 
if $S$ is the limit of a subsequence $(\tilde{\phi}_{k_l}(X))_{l\in \bb{N}}$, then $S\in \Omega.$
Let $S'$ be a weak* cluster point of the sequence 
$(\tilde{\psi}_{k_l}(\tilde{\phi}_{k_l}(X)))_{l\in \bb{N}}$; 
then $S'' = S - S'$ is a weak* cluster point of
$(\tilde{\psi}_{k_l}^{\perp}(\tilde{\phi}_{k_l}(X)))_{l\in \bb{N}}$.
By Lemma \ref{l_pre},
$$ \tilde{\psi}_k(\tilde{\phi}_k(X)) \in 
\overline{(\cl M_k\cap  \ran\psi_k) \bar \otimes \cl U 
+((\cap _{l=1}^{n-1}\cl A_l)\cap (\cl M _k\cap\ran\psi_k))\bar \otimes \cl V},$$
while
$$ \tilde{\psi}_k^{\perp}(\tilde{\phi}_k(X)) 
= \tilde{\phi}_k(\tilde{\psi}_k^{\perp}(X)) \in 
\cl B\bar\otimes\cl V\subseteq \Omega, \ \ k\in \bb{N}.$$

On the other hand, Lemma \ref{vitakyklos} implies that 
$$S'\in 
\overline{((\cap_k \cl M_k)\cap  \cl N) \bar \otimes \cl U 
+((\cap _{l=1}^{n-1}\cl A_l)\cap (\cap_k \cl M_k)\cap \cl N)\bar \otimes \cl V} \subseteq \Omega.$$
Thus, $S = S' + S''\in \Omega$.
\end{proof}

\begin{lemma}\label{lgama}
Proposition \ref{gama} holds if $r = 1$.
\end{lemma}
\begin{proof} 
Set $\cl B = \cl B_1$ and $\cl V = \cl U_1$,
let 
$\Omega = \overline{(\cl M\cap \cl C)\bar \otimes \cl U+\cl B\bar \otimes \cl V}$
and fix 
$$X\in  \overline{\cl M\bar \otimes \cl U+\cl B\bar \otimes \cl V} \cap \overline{ \cl C\bar \otimes
\cl U+\cl B\bar \otimes \cl V }.$$
Let $(\psi_k,\theta_k,\ran\psi_k,\cl Z_k)_{k\in \bb{N}}$ be a decomposition scheme for $\cl C$.
We have that 
$$\tilde{\psi}_k^\bot (X)\in \overline{\cl M\bar\otimes \cl U+\cl B \bar\otimes \cl V}
\cap \overline{\cl Z_k \bar\otimes \cl U+\cl B \bar\otimes \cl V}.$$
By Lemma \ref{l_o2}, 
$$ \tilde{\psi}_k^\bot (X)
\in \overline{(\cl M \cap \cl Z_k)\bar \otimes \cl U+\cl B\bar \otimes \cl V}\subseteq
\overline{( \cl M\cap \cl C)\bar \otimes \cl U+\cl B\bar \otimes \cl V }.$$
On the other hand, by Lemmas \ref{l_pre} and \ref{l_o2}, 
$$\tilde{\psi}_k(X)\in 
\overline{(\cl M_k\cap \ran \psi _k)\bar \otimes \cl U+\cl B\bar \otimes \cl V}$$ 
for all $k\in \bb{N}$.
Lemma \ref{4kyklos}
shows that, if 
$S$ is a weak* cluster point of the sequence $(\tilde{\psi}_k(X))_{k\in \bb{N}}$,
then 
$$S \in 
\overline{(\cap _{k} (\cl M_k\cap \ran \psi_k)) \bar \otimes \cl U+\cl B\bar \otimes \cl V}.$$
Since $\cap _{k} (\cl M_k\cap \ran \psi_k) \subseteq \cl M \cap \cl C$, we conclude that
$S \in \Omega$. The proof is complete.
\end{proof}

\begin{lemma}\label{ldelta}
Proposition \ref{delta} holds if $r = 1$.
\end{lemma}
\begin{proof}
Write $\cl B = \cl B_1$ and $\cl V = \cl U_1$. 
We use induction on the width $n$ of $\cl C.$ If $n=1$, the statement reduces to 
Lemma \ref{lgama}.
Suppose that the statement holds for all masa-bimodules $\cl C$ of width at most $n-1$
and let
$\cl C=\cap _{l=1}^n\cl C_l,$ where $\cl C_l$ is a nest algebra bimodule, $l = 1,\dots,n$. 
Fix
$$X\in \overline{\cl M\bar \otimes \cl U+\cl B\bar \otimes \cl V} \cap \overline{ \cl C\bar \otimes \cl U+\cl B\bar \otimes \cl V}.$$
Let $(\psi_k,\theta_k,\ran\psi_k,\cl Z_k)_{k\in \bb{N}}$ be a decomposition scheme for $\cl C_n$ 
and recall that $\cl N = \cap_{k}\ran\psi_k$. 
There exists a descending sequence $(\cl M_k)_{k\in \bb{N}}$ of ranges of contractive Schur idempotents 
such that $\cl M=\cap_{k\in \bb{N}} \cl M_k.$
Observe that, by the inductive assumption, 
\begin{eqnarray*}
\tilde{\psi}_k^\bot (X) & \in & \overline{\cl M\bar \otimes \cl U+\cl B\bar \otimes \cl V} \cap \overline{(\cap _{l=1}^{n-1} \cl C_l)\bar \otimes \cl U+\cl B\bar \otimes \cl V } 
 \cap \overline{\cl Z_k\bar \otimes \cl U+\cl B \bar \otimes \cl V}\\
& = &
\overline{(\cl M\cap (\cap _{l=1}^{n-1} \cl C_l)\bar \otimes \cl U+\cl B\bar \otimes \cl V } \cap
\overline{(\cl Z_k\bar \otimes \cl U+\cl B \bar \otimes \cl V)}.
\end{eqnarray*}
By Lemma \ref{l_o2}, 
$$\tilde{\psi}_k^\bot (X) \in \overline{ (  \cl M\cap \cl C)\bar \otimes \cl U+\cl B\bar \otimes \cl V }. $$

Observe that
$$\tilde{\psi}_k(X)\in \overline{\cl M_k\bar \otimes \cl U+\cl B\bar \otimes \cl V} \cap \overline{
(\cap _{l=1}^{n-1} \cl C_l)\bar \otimes \cl U+\cl B\bar \otimes \cl V } 
 \cap \overline{\ran \psi _k\bar \otimes \cl U+\cl B \bar \otimes \cl V}$$
and so, by Lemma \ref{l_o2}, 
$$\tilde{\psi}_k(X) \in 
\overline{(\cl M_k\cap \ran \psi _k)\bar \otimes \cl U+\cl B\bar \otimes \cl V} \cap \overline{
(\cap _l^{n-1} \cl C_l)\bar \otimes \cl U+\cl B\bar \otimes \cl V }, \ \ k\in \bb{N}.$$
By Lemma \ref{4kyklos}, if $S$ is a weak* cluster point of $(\tilde{\psi}_k(X))_{k\in\bb{N}}$, 
then
$$S\in \overline{(\cl M\cap \cl N)\bar \otimes \cl U+\cl B\bar \otimes \cl V }\cap \overline{
(\cap _{l=1}^{n-1} \cl C_l)\bar \otimes \cl U+\cl B\bar \otimes \cl V }.  $$
By the inductive assumption,
$$S\in \overline{(\cl M\cap \cl N \cap (\cap _{l=1}^{n-1} \cl C_l))\bar \otimes \cl U +
\cl B\bar \otimes \cl V } \subseteq \overline{(\cl M\cap \cl C)\bar \otimes \cl U + \cl B\bar \otimes \cl V},$$
since $\cl N\subseteq \cl C_n$. 
The proof is complete.
\end{proof}

\begin{lemma}\label{lepsilon}
Theorem \ref{epsilon} holds if $r = 1$.
\end{lemma}
\begin{proof}
Since each masa-bimodule of finite width is the finite intersection of 
nest algebra masa-bimodules, we may assume, without loss of generality, that 
$\cl C_j$ is a nest algebra bimodule, $j = 1,\dots,m$.
Set $\cl B = \cl B_1$, $\cl V = \cl U_1$ and $\cl C = \cap_{j=1}^m\cl C_j$. 
We use induction on $m.$ 
For $m = 1$, the statement is trivial; suppose it holds if the 
number of given bimodules is at most $m-1$
and fix 
$X\in \cap _{j=1}^m\overline{\cl C_j\bar \otimes \cl U+\cl B\bar \otimes \cl V}.$
Let $(\phi_k,\theta_k,\cl M_k,\cl Z_k)_{k\in \bb{N}}$ be a decomposition scheme for $\cl C_m$
and set $\cl M = \cap_{k\in \bb{N}}\cl M_k$. 
Since $X\in \overline{\cl C_m \bar\otimes \cl U  + \cl W}$,
Lemma \ref{l_pre}
implies that 
$$\tilde{\phi}_k^{\bot}(X) \in \overline{\cl Z_k\bar\otimes \cl U + \cl W}.$$
By the inductive assumption and the invariance of 
$\cap_{j=1}^{m-1}\cl C_j$ under $\phi_k$, we have that
$$\tilde{\phi}_k^{\bot}(X) \in \overline{(\cap_{j=1}^{m-1} \cl C_j) \bar\otimes \cl U + \cl W}.$$
Hence, by Lemma \ref{l_o2}, 
$$\tilde{\phi}_k^{\bot}(X) \in
\overline{((\cap _{i=1}^{m-1}\cl C_i)\cap \cl Z_k)\bar \otimes \cl U+\cl B\bar \otimes \cl V}\subseteq
\overline{\cl C\bar \otimes \cl U+\cl B\bar \otimes \cl V}, \ \ \ k\in \bb{N}.$$

On the other hand,
$$ \tilde{\phi}_k(X) \in \overline{\cl M _k \bar \otimes \cl U+\cl B\bar \otimes \cl V}\cap
\overline{(\cap _{j=1}^{m-1}\cl C_j)\bar \otimes \cl U+\cl B\bar \otimes \cl V},  \ \ \ k\in \bb{N}.$$
If $S$ is a weak* cluster point of the sequence $(\tilde{\phi} _k(X))_{k\in \bb{N}}$, 
by Lemma \ref{4kyklos} we have 
 $$S\in \overline{\cl M \bar \otimes \cl U + \cl B\bar \otimes \cl V}\cap
\overline{(\cap _{j=1}^{m-1}\cl C_j)\bar \otimes \cl U + \cl B\bar \otimes \cl V}.$$
By Lemma \ref{ldelta}, 
$$S\in \overline{(\cl M \cap (\cap _{i=1}^{m-1}\cl C_i))\bar \otimes \cl U + \cl B\bar \otimes \cl V}\subseteq
\overline{\cl C\bar \otimes \cl U + \cl B\bar \otimes \cl V}.$$
\end{proof}

We next establish the induction step for the 
proofs of Propositions \ref{mazi} -- \ref{delta} and Theorem \ref{epsilon};
this is done in Lemmas \ref{diokyklos2} -- \ref{lepsilon2} below.
To this end,
we assume that the statements in Proposition \ref{mazi} (i)--(iv) hold if
the space $\cl W$ has $r-1$ summands,
while Proposition \ref{mazi} (v), Proposition \ref{gama}, Proposition \ref{delta} and
Theorem \ref{epsilon} hold if $\cl W$ has $r$ summands.

\begin{lemma}\label{diokyklos2}
Proposition \ref{mazi} (i) holds if the space $\cl W$ has $r$ terms.
\end{lemma}
\begin{proof}
Let 
$$\Omega = \overline{(\cl N\cap \cl R)\bar \otimes \cl U + \cl R \bar \otimes \cl V + \cl W}$$ 
and fix
$$X\in \cap _{k,i} \overline{(\ran \psi _i\cap \ran \phi_k )\bar \otimes \cl U + \ran \phi_k \bar \otimes \cl V + \cl W}.$$
By Lemma \ref{l_pre}, 
$\tilde{\psi}_i^\bot (X)\in \overline{\ran \phi _k\bar \otimes \cl V + \cl W}$ for all $k,i\in \bb{N}.$
By the inductive assumption concerning Proposition \ref{mazi} (v),
$$\tilde{\psi}_i^\bot (X)\in \overline{\cl R \bar\otimes \cl V + \cl W} \subseteq \Omega, \ \ \ \  i\in \bb{N}.$$
On the other hand, for all $k\in \bb{N}$ we have
\begin{eqnarray*}
\tilde{\psi}_k(X) & \in & 
\overline{(\ran\psi_k\cap\ran\phi_k) \bar\otimes \cl U + 
(\ran \psi_k\cap \ran \phi_k) \bar \otimes \cl V + \cl W}\\
& = & 
\overline{(\ran \psi _k\cap \ran \phi_k ) \bar \otimes (\overline{\cl U+\cl V}) + \cl W}.
\end{eqnarray*}
Let $S$ be a weak* cluster point of $(\tilde{\psi}_k (X))_{k\in \bb{N}}$.
Once again by the inductive assumption concerning Proposition \ref{mazi} (v),
$$S \in \overline{(\cl N \cap \cl R) \bar \otimes (\overline{\cl U+\cl V}) + \cl W}.$$
Thus, $S\in \Omega$
and the proof is complete.
\end{proof}

\begin{lemma}\label{triakyklosnn}
Proposition \ref{mazi} (ii) holds if the space $\cl W$ has $r$ terms.
\end{lemma}
\begin{proof}
Let 
$$\Omega  = \overline{\cl N \bar \otimes \cl U + \cl B\bar \otimes \cl V + \cl W}$$ 
and fix
$$X\in \cap_i \overline{\ran \psi _i\bar \otimes \cl U+\cl B\bar \otimes \cl V + \cl W}.$$
Let $(\phi_k,\theta_k,\ran\phi_k,\cl Z_k)_{k\in \bb{N}}$ be a decomposition scheme for $\cl B$ and
observe that, by Lemma \ref{l_pre}, we have 
$$\tilde{\phi}_k(X) \in \cap_i \overline{(\ran \psi _i\cap \ran \phi _k)\bar \otimes \cl U+
\ran \phi _k\bar \otimes \cl V + \cl W}.$$
Letting $S$ be a weak* cluster point of the sequence 
$(\tilde{\phi}_k(X))_{k\in \bb{N}}$, we have that
$$S\in \cap_{k,i} \overline{(\ran \psi _i\cap \ran \phi _k)\bar \otimes \cl U+
\ran \phi _k\bar \otimes \cl V + \cl W}.$$
By Lemma \ref{diokyklos2}
and the fact that $\cl R = \cap_k \ran\phi_k \subseteq \cl B$, we have that
$S\in  \Omega.$ 
So it suffices to prove
that $\tilde{\phi}_k^\bot (X) \in \Omega $ for all $k\in \bb{N}$. Note that, by Lemma \ref{l_pre},
$$\tilde{\phi}_k^\bot (X)\in \cap _i\overline{ \ran \psi _i \bar \otimes \cl U+\cl Z_k\bar \otimes \cl V + \cl W}.$$
By the inductive assumption concerning Proposition \ref{mazi} (v) and
Lemma \ref{l_pre} again,
$$\tilde{\theta}_k^\bot (\tilde{\phi}_k^\bot (X))\in \cap_i\overline{\ran \psi _i\bar \otimes \cl U + \cl W} =
\overline{\cl N \bar \otimes \cl U + \cl W} \subseteq \Omega.$$
Write $\theta_k = \sum_{p=1}^k \theta_{k,p}$, 
where each $\theta_{k,p}$ is a contractive Schur idempotent whose range is contained in 
$\cl B$.
We have that 
$$\tilde{\theta}_{k,p}(\tilde{\phi}_k^\bot (X))\in
\cap _i\overline{(\ran \psi _i \cap \ran\theta_{k,p}) )\bar \otimes \cl U +
(\ran\theta_{k,p}) \bar \otimes \cl V + \cl W}.$$ By Lemma \ref{diokyklos2} (applied in the case 
of a constant sequence of maps with term $\theta_{k,p}$), we have 
$$\tilde{\theta}_{k,p}(\tilde{\phi}_k^\bot (X))\in \overline{\cl N \bar \otimes \cl U +
(\ran\theta_{k,p}) \bar \otimes \cl V + \cl W}\subseteq \Omega, \ \ p = 1,\dots,k.$$
It follows that 
$\tilde{\theta}_k(\tilde{\phi}_k^\bot (X)) \in \Omega$ and hence
$\tilde{\phi}_k^\bot (X) = \tilde{\theta}_k(\tilde{\phi}_k^\bot (X)) + \tilde{\theta}_k^{\bot}(\tilde{\phi}_k^\bot (X))\in \Omega$. 
\end{proof}

\begin{lemma}\label{alphakyklos2}
Proposition \ref{mazi} (iii) holds if the space $\cl W$ has $r$ terms.
\end{lemma}
\begin{proof}
We let 
$$\Omega = \overline{\cl N \bar \otimes \cl U+(\cl N \cap \cl B)\bar \otimes \cl V + \cl W}$$ 
and fix
$$X\in \cap_i \overline{\ran \psi_i\bar \otimes \cl U + (\ran \psi _i\cap \cl B)\bar \otimes \cl V + \cl W}.$$
Let $(\phi_k,\theta_k,\cl M_k,\cl Z_k)_{k\in \bb{N}}$ be a decomposition scheme for $\cl B$ and  
observe that
$$\tilde{\phi}_k^\bot(X)\in \cap _i\overline{\ran \psi _i\bar \otimes \cl U + (\ran \psi _i\cap \cl Z_k)
\bar \otimes \cl V + \cl W}.$$
Using the inductive assumption concerning Proposition \ref{mazi} (v), we have
$$\tilde{\theta}_k^\bot(\tilde{\phi}_k^\bot(X)) \in \cap _i \overline{\ran \psi_i \bar \otimes \cl U + \cl W}
= \overline{\cl N \bar \otimes \cl U + \cl W}\subseteq \Omega.$$
Write
$\theta_k = \sum_{p=1}^k \theta_{k,p}$ as in the proof of Lemma \ref{triakyklosnn};
then
\begin{eqnarray*}
\tilde{\theta}_{k,p}(\tilde{\phi}_k^\bot(X)) 
& \in & \cap _i\overline{(\ran \psi _i \cap \ran\theta_{k,p})\bar \otimes \cl U 
+ (\ran \psi _i\cap \ran \theta_{k,p})
\bar \otimes \cl V + \cl W}\\
& = & 
\cap _i\overline{(\ran \psi _i \cap \ran\theta_{k,p})\bar \otimes (\cl U 
+ \cl V) + \cl W}.
\end{eqnarray*}
By the inductive assumption concerning Proposition \ref{mazi} (v), we have
that $\tilde{\theta}_{k,p}$ $(\tilde{\phi}_k^\bot(X))$ $\in$ $\Omega$ for each 
$p = 1,\dots,k$. It follows that 
$\tilde{\phi}_k^\bot (X) \in \Omega .$
Let $S$ be a weak* cluster point of the sequence $(\tilde{\phi_k}(X))_{k\in \bb{N}}$.
Since
$$\tilde{\phi}_k (X) \in \cap _k \overline{(\ran \psi _k\cap \cl M_k)\bar \otimes
\overline{\cl U+\cl V} + \cl W}, $$ 
by the same inductive assumption once again,  
$$S\in \overline{(\cl N \cap(\cap _k \cl M_k))\bar \otimes
\overline{\cl U+\cl V} + \cl W}\subseteq \Omega .$$
\end{proof}

\begin{lemma}\label{vitakyklos2}
Proposition \ref{mazi} (iv) holds if the space $\cl W$ has $r$ terms.
\end{lemma}
\begin{proof}
Let 
$$\Omega = \overline{\cl N  \bar \otimes \cl U+(\cl N \cap \cl B)\bar \otimes \cl V + \cl W}.$$ 
We use induction on the width $n$ of $\cl B.$ The case $n=1$ reduces to Lemma \ref{alphakyklos2}.
Suppose that the statement holds for masa-bimodules of width at most $n-1$ and let
$\cl B = \cap _{l=1}^n\cl C_l$ where every $\cl C_l$ is nest algebra bimodule, $l = 1,\dots,n$. 
Fix
$$X\in  \cap_i \overline{ \ran \psi _i\bar \otimes \cl U + (\ran \psi _i\cap \cl B)\bar \otimes \cl V + \cl W}.$$
By the inductive assumption, $X$ belongs to the intersection of
$$\overline{\cl N 
\bar \otimes \cl U + (\cl N \cap (\cap _{l=1}^{n-1}\cl C_l))\bar \otimes \cl V + \cl W} \ 
\mbox{ and } \ 
\overline{\cl N \bar \otimes \cl U + (\cl N \cap \cl C_n)\bar \otimes \cl V + \cl W}.$$
Let $(\phi_k,\theta_k,\cl M_k,\cl Z_k)_{k\in \bb{N}}$ be a decomposition scheme for $\cl C_n$.
For a fixed $k\in \bb{N}$, we have that $\tilde{\phi}_k^\bot (X) $ belongs to the intersection of
$$\overline{\cl N
\bar \otimes \cl U + (\cl N \cap (\cap _{l=1}^{n-1}\cl C_l))\bar \otimes \cl V + \cl W} \ 
\mbox{ and } 
\overline{\cl N \bar \otimes \cl U + (\cl N \cap \cl Z_k)\bar \otimes \cl V + \cl W}.$$
By Lemma \ref{l_o2}, 
$$\tilde{\phi}_k^\bot (X) \in 
\overline{ \cl N \bar \otimes \cl U + 
(\cl N \cap (\cap _{i=1}^{n-1}\cl C_i)\cap \cl Z_k)\bar \otimes \cl U + \cl W}
\subseteq \Omega. $$
Let $S$ be a weak* cluster point of the sequence $(\tilde{\phi}_k(X))_{k\in \bb{N}}$; 
we have
$$S\in \cap _{k}\overline{(\ran \psi_k\cap \cl M_k) \bar \otimes \cl U
+ (\ran \psi _k\cap \cl M_k \cap (\cap _{l=1}^{n-1}\cl C_l))
\bar \otimes \cl V  + \cl W}.$$ 
By the inductive assumption, the latter space is equal to
$$\overline{(\cl N \cap (\cap _{k} \cl M_k) ) \bar \otimes \cl U+(\cl N \cap (\cap _k\cl M_k)
\cap (\cap _{l=1}^{n-1}\cl C_l))\bar \otimes \cl V + \cl W}$$ which is contained in $\Omega$ since
$\cap _k\cl M_k \subseteq \cl C_n$.
\end{proof}

\begin{lemma}\label{4kyklos2}
Proposition \ref{mazi} (v) holds if the space $\cl W$ has $r+1$ terms.
\end{lemma}
\begin{proof}
Let
$\Omega = \overline{\cl N \bar \otimes \cl U + \cl W}$.
Set $\cl B = \cl B_{r+1}$, $\cl V = \cl U_{r+1}$ and $\cl W_0 = \sum_{i=1}^r \cl B_i\bar\otimes\cl U_i$.
We use induction on the width $n$ of $\cl B$.
For $n=1$ the conclusion 
follows from the inductive assumption concerning Proposition \ref{mazi} (ii). 
Assume that it holds when the length of $\cl B$ does not exceed $n-1$, and 
suppose that $\cl B=\cap _{l=1}^n\cl C_l$
where $\cl C_l$  is a nest algebra bimodule, $l = 1,\dots,n$.
Let $(\phi_k,\theta_k,\cl M_k,\cl Z_k)_{k\in \bb{N}}$ be a decomposition scheme for $\cl C_n$.
Fix $$X\in \cap_i \overline{\ran \psi _i\bar \otimes \cl U + \cl B\bar \otimes \cl V + \cl W_0}.$$
By assumption,
$$X\in \overline{\cl N \bar \otimes \cl U + (\cap _{l=1}^{n-1}\cl C_l)\bar \otimes \cl V + \cl W_0} 
\cap 
\overline{\cl N \bar \otimes \cl U + \cl C_n\bar \otimes \cl V + \cl W_0}.$$
For a fixed $k\in \bb{N}$, we have 
$$\tilde{\phi}_k^\bot (X) \in \overline{\cl N \bar \otimes \cl U + (\cap _{l=1}^{n-1}\cl C_l)\bar \otimes \cl V + \cl W_0}
\cap \overline{\cl N \bar\otimes \cl U + \cl Z_k\bar \otimes \cl V + \cl W_0}.$$
By Lemma \ref{l_o2}, $\tilde{\phi}_k^\bot (X) \in \Omega .$
Let $S$ be a weak* cluster point of the sequence $(\tilde{\phi}_k(X))_{k\in\bb{N}}$,
and $S'$ and $S''$ be weak* cluster points of 
$(\tilde{\psi}_k^{\bot}(\tilde{\phi}_k(X)))_{k\in \bb{N}}$ and 
$(\tilde{\psi}_k(\tilde{\phi}_k$ $(X)))_{k\in \bb{N}}$, respectively, such that $S = S' + S''$.

We have 
$$\tilde{\phi}_k(X) \in \overline{(\cl M_k \cap  \cl N) \bar \otimes \cl U 
+ ((\cap _{l=1}^{n-1}\cl C_l)\cap \cl M_k)\bar \otimes \cl V + \cl W_0}.$$
It follows that
$$\tilde{\psi}_k^\bot (\tilde{\phi}_k(X) )\in 
\overline{((\cap _{l=1}^{n-1}\cl C_l)\cap \cl M_k)\bar \otimes \cl V + \cl W_0}, 
\ \ k\in \bb{N},$$
and hence, by the inductive assumption concerning Proposition \ref{mazi} (v),
$$S' \in \overline{(\cap _{l=1}^{n-1}\cl C_l) \bar \otimes \cl V + \cl W_0} \cap
\overline{(\cap_{k}\cl M_k) \bar \otimes \cl V + \cl W_0}.$$
It follows from
the inductive assumption concerning Proposition \ref{delta} that
$S' \in \Omega$.
On the other hand, 
$$S'' \in \cap _{k}\overline{(\cl M_k\cap \ran \psi _k)\bar \otimes \cl U +
((\cap _{l=1}^{n-1}\cl C_l)\cap \cl M_k\cap \ran \psi _k)\bar \otimes\cl V + \cl W_0}.$$
Since $\|\phi_k\psi_k\|\leq 1$ for each $k\in \bb{N}$, 
Lemma \ref{vitakyklos2} implies that $S''\in \Omega$. It now follows that $S\in \Omega$. 
\end{proof}

\begin{lemma}\label{lgama2}
Proposition \ref{gama} holds if the space $\cl W$ has $r+1$ terms.
\end{lemma}
\begin{proof}
Let $\Omega = \overline{(\cl M\cap \cl C)\bar \otimes \cl U + \cl W}$ and 
fix 
$$X\in \overline{\cl M\bar \otimes \cl U + \cl W}\cap
\overline{ \cl C\bar \otimes \cl U + \cl W}.$$
Let $(\psi_k,\theta_k,\ran\psi_k,\cl Z_k)_{k\in \bb{N}}$ be a decomposition scheme for $\cl C$.
Let $(\phi_k)_{k\in \bb{N}}$ be a nested sequence  of
contractive Schur idempotents such that $\cl M = \cap_{k=1}^{\infty} \cl M_k$,
where $\cl M_k = \ran\phi_k$, $k\in \bb{N}$.
We first observe that, by Lemma \ref{l_o2}, 
$$\tilde{\theta}_k(X) \in 
\overline{(\cl Z_k \cap \cl M) \bar \otimes \cl U + \cl W}\subseteq \Omega, \ \ \ 
k\in \bb{N}.$$
On the otehr hand, by Lemma \ref{l_pre}, 
$$(\tilde{\psi}_k + \tilde{\theta}_k)^{\perp}(X) \in \cl W\subseteq \Omega
\ \mbox{ and } \ \tilde{\phi}_k^{\perp}(X)\in \cl W\subseteq \Omega, \ \ \ k\in \bb{N}.$$ 
It follows that 
$\tilde{\phi}_k^{\perp}(\tilde{\psi}_k(X)) = \tilde{\psi}_k(\tilde{\phi}_k^{\perp}(X))\in \Omega$
for each $k$.
Let   $S$ be a weak* cluster point of 
$(\tilde{\phi}_k (\tilde{\psi}_k(X)))_{k\in \bb{N}}$. 
Since 
$$\tilde{\phi}_k(\tilde{\psi}_k(X))\in \overline{(\cl M_k\cap \ran \psi _k)\bar \otimes \cl U + \cl W}$$ for all $k\in \bb{N}$, we have, by Lemma \ref{4kyklos2}, that $S$ belongs to
$\overline{\cap _{k} (\cl M_k\cap \ran \psi _k)\bar \otimes \cl U + \cl W}$,
which is a subset of $\Omega$.
Since
$$X = \tilde{\phi}_k(\tilde{\psi}_k(X)) + \tilde{\phi}_k^{\perp}(\tilde{\psi}_k(X))
+ \tilde{\theta}_k(X) + (\tilde{\psi}_k + \tilde{\theta}_k)^{\perp}(X), \ \ \ k\in \bb{N},$$
it now follows that
$X\in \Omega$. 
\end{proof}

\begin{lemma}\label{ldelta2}
Proposition \ref{delta} holds if the space $\cl W$ has $r+1$ terms.
\end{lemma}
\begin{proof}
We use induction on the width $n$ of $\cl C.$ If $n=1$ the statement reduces to Lemma \ref{lgama2}.
Suppose that the statement holds for all masa-bimodules $\cl C$ of width not exceeding $n-1$
and let $\cl C = \cap _{l=1}^n\cl C_l,$ where $\cl C_l$ is a nest algebra bimodule,
$l = 1,\dots,n$. Fix
$$X\in \overline{\cl M\bar \otimes \cl U + \cl W} \cap \overline{ \cl C\bar \otimes \cl U + \cl W}$$
and let $(\psi_k,\theta_k,\ran\psi_k,\cl Z_k)_{k\in \bb{N}}$ be a decomposition scheme for $\cl C_n$.
We also assume that $\cl M = \cap _k\cl M_k$ where every $(\cl M_k)_{k\in \bb{N}}$ is
descending sequence of ranges of contractive Schur idempotents. 
Using Lemma \ref{l_o2} and the inductive assumption, we obtain
\begin{eqnarray*}
\tilde{\psi}_k^\bot (X) & \in & \overline{\cl M\bar \otimes \cl U + \cl W}\cap
\overline{(\cap_{l=1}^{n-1} \cl C_l)\bar \otimes \cl U + \cl W}
\cap \overline{\cl Z_k\bar \otimes \cl U+\cl B \bar \otimes \cl V}\\
& = & \overline{\cl M\cap (\cap _{l=1}^{n-1} \cl C_l)\bar \otimes \cl U + \cl W} \cap
\overline{\cl Z_k\bar \otimes \cl U + \cl W}\\
& = & \overline{(\cl M\cap (\cap _{l=1}^{n-1} \cl C_l) \cap \cl Z_k) \bar \otimes \cl U + \cl W} 
\subseteq \overline{(\cl M\cap \cl C) \bar \otimes \cl U + \cl W}.
\end{eqnarray*}
Using Lemma \ref{l_o2} again, we have
\begin{eqnarray*}
\tilde{\psi}_k(X) & \in & \overline{\cl M_k\bar \otimes \cl U + \cl W}\cap
\overline{(\cap _{l=1}^{n-1} \cl C_l)\bar \otimes \cl U + \cl W}
\cap \overline{\ran \psi _k\bar \otimes \cl U + \cl W}\\
& = & \overline{(\cl M_k\cap \ran \psi _k)\bar \otimes \cl U + \cl W}
\cap \overline{(\cap_{l=1}^{n-1} \cl C_l)\bar \otimes \cl U + \cl W},
\end{eqnarray*}
for every $k\in \bb{N}$.
By Lemma \ref{4kyklos2}, if $S$ is a weak* cluster point of $(\tilde{\psi}_k(X))_{k\in \bb{N}}$,
then
$$S\in \overline{(\cl M\cap \cl N)\bar \otimes \cl U + \cl W}
\cap \overline{(\cap _{l=1}^{n-1} \cl C_l)\bar \otimes \cl U + \cl W}.$$
By the inductive assumption,
$$S\in \overline{(\cl M\cap \cl N \cap (\cap _{l=1}^{n-1} \cl C_l))\bar \otimes \cl U + \cl W}
\subseteq \overline{(\cl M\cap \cl C)\bar \otimes \cl U + \cl W}.$$
The proof is complete. 
\end{proof}

\begin{lemma}\label{lepsilon2}
Theorem \ref{epsilon} holds if the space $\cl W$ has $r+1$ terms.
\end{lemma}
\begin{proof}
It suffices to prove that if $\cl C_1,\dots,\cl C_m$ are weak* closed nest algebra bimodules 
and $\cl C = \cap_{i=1}^m \cl C_i$, then 
$$\cap _{i=1}^m\overline{\cl C_i\bar \otimes \cl U+\cl W} = \overline{\cl C\bar \otimes \cl U+\cl W},$$ 
where $\cl  C=\cap _{i=1}^m \cl C_i$.
We use induction on $m.$ Suppose that
$$\cap _{i=1}^{m-1}\overline{\cl C_i\bar \otimes \cl U+\cl W} =
\overline{(\cap _{i=1}^{m-1}\cl C_i)\bar \otimes \cl U+\cl W}$$
and fix $X\in \cap _{i=1}^m\overline{\cl C_i\bar \otimes \cl U+\cl W}.$
Let $(\psi_k,\theta_k,\cl M_k,\cl Z_k)_{k\in \bb{N}}$ be a decomposition scheme for $\cl C_m$.
Using the inductive assumption and Lemmas \ref{l_pre} and \ref{l_o2}, we have
$$\tilde{\psi}_k^\bot (X) \in \overline{((\cap _{i=1}^{m-1}\cl C_i)\cap \cl Z_k)\bar \otimes \cl U+\cl W}\subseteq
\overline{\cl C\bar \otimes \cl U+\cl W}.$$
On the other hand,
$$\tilde{\psi}_k(X)\in \overline{\cl M _k \bar \otimes \cl U+\cl W}\cap
\overline{(\cap _{i=1}^{m-1}\cl C_i)\bar \otimes \cl U+\cl W}.$$
Thus, if $S$ is a weak* cluster point of the sequence $(\psi _k(X))_{k\in \bb{N}}$ then, by Lemma \ref{4kyklos2},
we have that
$$S\in \overline{\cl M \bar \otimes \cl U+\cl W}\cap
\overline{(\cap _{i=1}^{m-1}\cl C_i)\bar \otimes \cl U+\cl W}.$$
By Lemma \ref{ldelta2},
$$S\in \overline{(\cl M \cap (\cap _{i=1}^{m-1}\cl C_i))\bar \otimes \cl U+\cl W}\subseteq
\overline{\cl C\bar \otimes \cl U+\cl W}.$$
The proof is complete.
\end{proof}

\bigskip

Lemmas \ref{diokyklos}--\ref{lepsilon2} conclude the proof of Propositions \ref{mazi}--\ref{delta} and Theorem \ref{epsilon}.
The following statement, which is 
an equivalent formulation of Theorem \ref{epsilon}, follows from that theorem
by a straightforward induction on $r$.

\begin{corollary}\label{epsilona}
Let $r,l_1,\dots,l_r\in \bb{N}$, 
$\{\cl B_j^i\}_{j=1}^{l_i}$ be a family of masa bimodules of finite width and
$\cl U_i$ be a weak* closed subspace of $\cl B(K_1,K_2)$, $i = 1,\dots,r$.
Set $\cl B^i = \cap_{j=1}^{l_i} \cl B_j^i$, $i = 1,\dots,r$.
Then
$$\bigcap_{j_1,\dots,j_r} \overline{\cl B_{j_1}^1 \bar \otimes \cl U_1 + \cdots + \cl B_{j_r}^r\bar \otimes \cl U_r}
 = \overline{\cl B^1\bar \otimes \cl U_1 + \cdots + \cl B^r \bar \otimes \cl U_r}.$$
\end{corollary}


\section{Operator synthesis of unions of products}\label{s_appos}

In this section we apply the results from Sections \ref{s_stab} and \ref{s_insum}
to study questions about operator synthesis. 
We start by recalling the main definitions regarding the notion of operator synthesis.

Let $(X_1,\mu_1)$ and $(X_2,\mu_2)$ be standard measure spaces, that is, the measures 
$\mu_1$ and $\mu_2$ are
regular Borel measures with respect to some Borel structures on $X_1$ and $X_2$ arising from
complete metrizable topologies.
Let $H_1 = L^2(X_1,\mu_1)$ and $H_2 = L^2(X_2,\mu_2)$. 
For a function $\nph\in L^{\infty}(X_1,\mu_1)$, let 
$M_{\nph}$ be the (bounded) operator on $H_1$ given by 
$M_\nph f = \nph f$, $f\in L^2(X_1,\mu_1)$; similarly define
$M_\psi$ for $\psi\in L^{\infty}(X_2,\mu_2)$. 
Let 
$$\cl D_1 = \{M_{\nph} : \nph\in L^{\infty}(X_1,\mu_1)\}.$$
We have that $\cl D_1$ is a masa; we define $\cl D_2\subseteq \cl B(H_2)$ similarly.
We need several facts and notions from the theory of masa-bimodules
\cite{a}, \cite{eks}, \cite{st}. A subset $E\subseteq X_1\times X_2$ is
called {\it marginally null} if $E\subseteq (M_1\times
X_2)\cup(X_1\times M_2)$, where $\mu_1(M_1) = \mu_2(M_2) = 0$. We call two
subsets $E,F\subseteq X_1\times X_2$ {\it marginally equivalent} (and
write $E\cong F$) if the symmetric difference of $E$ and $F$ is
marginally null. A set $\kappa\subseteq X_1\times X_2$ is called {\it
$\omega$-open} if it is marginally equivalent to a (countable) union
of the form $\cup_{i=1}^{\infty} \alpha_i\times\beta_i$, where
$\alpha_i\subseteq X_1$ and $\beta_i\subseteq X_2$ are measurable, $i\in
\bb{N}$. The complements of $\omega$-open sets are called {\it
$\omega$-closed}. An operator $T\in \cl B(H_1,H_2)$ is said to be
supported on $\kappa$ if $M_{\chi_{\beta}}TM_{\chi_{\alpha}} = 0$
whenever $(\alpha\times\beta)\cap \kappa \cong \emptyset$. (Here
$\chi_{\gamma}$ stands for the characteristic function of a measurable subset $\gamma$.) 
Given an $\omega$-closed set $\kappa\subseteq X_1\times X_2$, let 
$$\frak{M}_{\max}(\kappa) = \{T\in \cl B(H_1,H_2) : T \mbox{ is supported on } \kappa\}.$$
The space $\frak{M}_{\max}(\kappa)$ is a \emph{reflexive} masa-bimodule in the sense that 
$\Ref{\frak{M}_{\max}(\kappa)} = \frak{M}_{\max}(\kappa)$ where,
for a subspace $\cl U\subseteq \cl B(H_1,H_2)$, we let its \emph{reflexive hull} \cite{ls} be the subspace
$$\Ref{\cl U} = \{T\in \cl B(H_1,H_2) : Tx\in \overline{\cl U x}, \ \mbox{ for all } x\in H_1\}.$$

We note two straightforward 
properties of the reflexive hull that will be used in the sequel: 
it is monotone ($\cl U_1\subseteq \cl U_2$
implies $\Ref{\cl U_1}\subseteq \Ref{\cl U_2}$) and idempotent ($\Ref{\Ref{\cl U}} = \Ref{\cl U}$). 

It was shown in \cite{eks} that every reflexive masa-bimodule is of the form 
$\frak{M}_{\max}(\kappa)$ for some,
unique up to marginal equivalence, $\omega$-closed set $\kappa\subseteq X\times Y$. 
If $\cl U$ is any masa-bimodule, then its \emph{support} $\supp \cl U$ is defined to be the 
$\omega$-closed set $\kappa\subseteq X\times Y$ such that 
$\Ref{\cl U} = \frak{M}_{\max}(\kappa)$.
The masa-bimodule $\frak{M}_{\max}(\kappa)$ is the largest,
with respect to inclusion, (weak* closed) masa-bimodule with support $\kappa$
(see \cite{eks}).
As an extension of Arveson's work on commutative subspace lattices \cite{a}, 
it was shown in \cite{st} that if $\kappa$ is an $\omega$-closed set, 
then there exists a smallest, with respect to inclusion, 
(weak* closed) masa-bimodule $\frak{M}_{\min}(\kappa)$
with support $\kappa$. 
The $\omega$-closed subset $\kappa\subseteq X\times Y$ is called \emph{operator synthetic}
if $\frak{M}_{\min}(\kappa) = \frak{M}_{\max}(\kappa)$. 
The roots of the notion of operator synthesis lie in Harmonic Analysis --
it is an operator theoretic version of the well-known 
concept of spectral synthesis. 
We refer the reader to \cite{a} for a relevant discussion, 
and to \cite{st} for the formal relation between the two concepts,
which will be briefly summarised at the end of the section.

The supports of masa-bimodules of finite width will be called \emph{sets of finite width}. 
A set $\kappa\subseteq X_1\times X_2$ is of finite width 
precisely when it is the set of solutions of a system of (finitely many)
measurable function inequalities, that is, precisely when 
it has the form
$$\kappa = \{(x,y)\in X_1\times X_2 : f_k(x)\leq g_k(y), k = 1,\dots,n\},$$
where $f_k : X_1\to \bb{R}$ and $g_k : X_2\to \bb{R}$ are 
measurable functions, $k = 1,\dots,n$ (see, {\it e.g.}, \cite{jlms}).  
It was shown in \cite{st} and \cite{jlms} that sets of finite width are operator synthetic.

In this section, we will be concerned with the question of when 
operator synthesis is preserved under unions of products. 
Suppose that $(Y_1,\nu_1)$, $(Y_2,\nu_2)$ is another pair of standard measure spaces, 
$K_i = L^2(Y_i,\nu_i)$ and $\cl C_i$ is the multiplication masa of $L^{\infty}(Y_i,\nu_i)$, 
$i = 1,2$.
Let 
$\cl U\subseteq \cl B(H_1,H_2)$ be a 
$\cl D_2,\cl D_1$-module and 
$\cl V\subseteq \cl B(K_1,K_2)$ be a $\cl C_2,\cl C_1$-module. 
Then the subspace $\cl U\bar\otimes \cl V$ is a 
$\cl D_2\bar\otimes \cl C_2,\cl D_1\bar\otimes\cl C_1$-module, and hence its 
support is a subset of $(X_1\times Y_1)\times (X_2\times Y_2)$. 
The \lq\lq flip''
$$\rho : (X_1\times X_2)\times (Y_1\times Y_2) \to (X_1\times Y_1)\times (X_2\times Y_2),$$ 
given by 
$$\rho(x_1,x_2,y_1,y_1) = (x_1,y_1,x_2,y_2), \ \ \ x_i\in X_i, y_i\in Y_i, i = 1,2,$$
is thus needed in order to relate 
$\supp(\cl U\bar\otimes\cl V)$ to $(\supp\cl U)\times (\supp \cl V)$. 
Indeed, it was shown in \cite{mgot} that 
\begin{equation}\label{equv}
\supp(\cl U\bar\otimes\cl V) = \rho(\supp\cl U\times \supp\cl V).
\end{equation}
It was observed in \cite[Lemma 4.19]{stt} that 
\begin{equation}\label{eqmin}
\frak{M}_{\min}(\rho(\kappa\times\lambda)) = \frak{M}_{\min}(\kappa)\bar\otimes\frak{M}_{\min}(\lambda),
\end{equation}
whenever $\kappa\subseteq X_1\times X_2$ and $\lambda\subseteq Y_1\times Y_2$
are $\omega$-closed sets.

\begin{remark} \label{r_indt}
{\rm If $\kappa\subseteq X_1\times X_2$ and $\lambda\subseteq Y_1\times Y_2$ are 
non-marginally null $\omega$-closed sets such that $\rho(\kappa\times\lambda)$ is operator synthetic,
then both $\kappa$ and $\lambda$ are operator synthetic. Indeed, suppose that 
$T\in \frak{M}_{\max}(\kappa)$, and let $0\neq S\in \frak{M}_{\min}(\lambda)$. Then 
$T\otimes S\in \frak{M}_{\max}(\rho(\kappa\times\lambda))$ and, by assumption
and identity (\ref{eqmin}), 
$$T\otimes S\in \frak{M}_{\min}(\kappa)\bar\otimes\frak{M}_{\min}(\lambda).$$
It now easily follows that $T\in \frak{M}_{\min}(\kappa)$. Thus, $\kappa$ is operator synthetic and by symmetry $\lambda$ is so as well.}
\end{remark}

\begin{remark}\label{r_ssynp}
{\rm
Let $G$ and $H$ be locally compact groups. 
A problem in Harmonic Analysis asks when, given closed sets 
$E\subseteq G$ and $F\subseteq H$ satisfying spectral synthesis, 
the set $E\times F$ satisfies spectral synthesis as a subset of the direct product 
$G\times H$. 
We refer the reader to \cite{eym} 
for the definition of the notion of spectral 
synthesis and other basic concepts and 
results from non-commutative Harmonic Analysis. 
Analogues of 
identities (\ref{equv}) and (\ref{eqmin}) in the setting of 
Harmonic Analysis can be formulated as follows. 
Let $\vn(G)\subseteq \cl B(L^2(G))$ 
(resp. $\vn(H)\subseteq \cl B(L^2(H))$) be the von Neumann 
algebra of $G$ (resp. $H$), and note that $\vn(G)\bar\otimes \vn(H)$ can be 
naturally identified with $\vn(G\times H)$. 
The Harmonic Analysis analogue of masa-bimodules are \emph{invariant spaces};
these are subspaces $\cl X\subseteq \vn(G)$ that are annihilators of 
ideals of the Fourier algebra $A(G)$ of $G$. 
Given an invariant space $\cl X\subseteq \vn(G)$, one may define 
its support $\supp\cl X$ as
the null set of its preannihilator in $A(G)$. 
It is not difficult to see 
that if $\cl X\subseteq \vn(G)$ and $\cl Y\subseteq \vn(H)$ are 
invariant spaces, then 
$\supp (\cl X\bar\otimes\cl Y) = (\supp\cl X) \times (\supp\cl Y)$
and that, given any closed subset $E\subseteq G$,
there exists a largest (resp. smallest) invariant space $\cl X_{\max}(E)$ 
(resp. $\cl X_{\min}(E)$) with support $E$, and 
$\cl X_{\min}(E)\bar\otimes \cl X_{\min}(F) = \cl X_{\min}(E\times F)$.
}
\end{remark}

The next proposition describes the connection between operator synthesis and 
tensor product formulas.

\begin{proposition}\label{1234}
Let $\cl U\subseteq\cl B(H_1,H_2)$ and $\cl V\subseteq \cl B(K_1,K_2)$ be 
masa-bimodules with supports 
$\kappa\subseteq X_1\times X_2$ and $\lambda\subseteq Y_1\times Y_2$,
respectively. Then 
\begin{equation}\label{eqf}
\supp \cl F(\cl U,\cl V) = \rho(\kappa\times\lambda)
\end{equation}
and 
\begin{equation}\label{eqm}
\cl F(\frak{M}_{\max}(\kappa),\frak{M}_{\max}(\lambda)) = 
\frak{M}_{\max}(\rho(\kappa\times\lambda)).
\end{equation}

Moreover, if $\kappa$ and $\lambda$ are operator synthetic, 
then the following statements are equivalent:

(i) \ \ $\rho(\kappa\times\lambda)$ is operator synthetic;

(ii) \ $\cl F(\frak{M}_{\max}(\kappa),\frak{M}_{\max}(\lambda)) = \frak{M}_{\max}(\kappa)\bar\otimes \frak{M}_{\max}(\lambda)$;

(iii) $\cl F(\frak{M}_{\min}(\kappa),\frak{M}_{\min}(\lambda)) = \frak{M}_{\min}(\kappa)\bar\otimes \frak{M}_{\min}(\lambda)$.
\end{proposition}
\begin{proof}
We have that 
$\cl F(\cl U,\cl V) = (\cl U\bar\otimes\cl B(K_1,K_2))\cap (\cl B(H_1,H_2)\bar\otimes\cl V)$, and hence 
$$\supp \cl F(\cl U,\cl V) = \supp(\cl U\bar\otimes\cl B(K_1,K_2))\cap \supp (\cl B(H_1,H_2)\bar\otimes\cl V).$$
By (\ref{equv}),
the support of 
$\cl U\bar\otimes\cl B(K_1,K_2)$ 
(resp. $\cl B(H_1,H_2)\bar\otimes\cl V$) is 
$\rho(\kappa\times (Y_1\times Y_2))$
(resp. $\rho((X_1\times X_2)\times \lambda)$).
Identity (\ref{eqf}) now readily follows.
To establish (\ref{eqm}) note that
$\cl F(\frak{M}_{\max}(\kappa),\frak{M}_{\max}(\lambda))$ and 
$\frak{M}_{\max}(\rho(\kappa\times\lambda))$ are both reflexive and, by (\ref{eqf}), 
have equal supports.

Suppose that $\kappa$ and $\lambda$ are operator synthetic. 

(ii)$\Leftrightarrow$(i) Using \cite[Lemma 4.19]{stt} for the first equality below
and identity (\ref{eqm}) for the last one, we have 
\begin{eqnarray*}
\frak{M}_{\min}(\rho(\kappa\times\lambda)) & = & \frak{M}_{\min}(\kappa)\bar\otimes\frak{M}_{\min}(\lambda)
=  \frak{M}_{\max}(\kappa)\bar\otimes\frak{M}_{\max}(\lambda)\\
& \subseteq &
\cl F(\frak{M}_{\max}(\kappa), \frak{M}_{\max}(\lambda))
=  \frak{M}_{\max}(\rho(\kappa\times\lambda)).
\end{eqnarray*}

If the inclusion in the above chain is equality then we have that 
$\frak{M}_{\min}(\rho(\kappa\times\lambda)) = \frak{M}_{\max}(\rho(\kappa\times\lambda))$, in other words, that 
$\rho(\kappa\times\lambda)$ is operator synthetic. Conversely, if 
$\rho(\kappa\times\lambda)$ is operator synthetic then we must have equalities throughout. 

(iii)$\Leftrightarrow$(i) follows similarly from the chain
\begin{eqnarray*}
\frak{M}_{\min}(\rho(\kappa\times\lambda)) & = & 
 \frak{M}_{\min}(\kappa)\bar\otimes\frak{M}_{\min}(\lambda) \subseteq  \cl F(\frak{M}_{\min}(\kappa),\frak{M}_{\min}(\lambda))\\
& = & (\frak{M}_{\min}(\kappa)\bar\otimes\cl B(K_1,K_2))\cap (\cl B(H_1,H_2)\bar\otimes \frak{M}_{\min}(\lambda))\\
& = & (\frak{M}_{\max}(\kappa)\bar\otimes\cl B(K_1,K_2))\cap (\cl B(H_1,H_2)\bar\otimes \frak{M}_{\max}(\lambda))\\
& = & \cl F(\frak{M}_{\max}(\kappa), \frak{M}_{\max}(\lambda))
= \frak{M}_{\max}(\rho(\kappa\times\lambda)).
\end{eqnarray*}
\end{proof}


\begin{corollary}\label{c_synps}
Let $\kappa\subseteq X_1\times X_2$ be an operator synthetic $\omega$-closed set. 
If $\frak{M}_{\max}(\kappa)$ has property $S_{\sigma}$ then 
$\rho(\kappa\times\lambda)$ is operator synthetic for every operator synthetic 
$\omega$-closed set $\lambda\subseteq Y_1\times Y_2$.
\end{corollary}
\begin{proof}
Immediate from Proposition \ref{1234} (ii)$\Leftrightarrow$(i).
\end{proof}

It follows from Corollary \ref{c_synps} that if $\kappa$ is a set of finite width then 
$\rho(\kappa\times\lambda)$ is operator synthetic whenever $\lambda$ is so. 
In fact, we have the following stronger result.

\begin{corollary}\label{c_morear}
Let $\kappa\subseteq X_1\times X_2$ and $\lambda\subseteq Y_1\times Y_2$ be operator synthetic sets
and $\kappa'\subseteq X_1\times X_2$ be an $\omega$-closed set of finite width. 
If $\rho(\kappa\times\lambda)$ is operator synthetic then so is 
$\rho((\kappa\cup \kappa')\times\lambda)$.
\end{corollary}
\begin{proof}
Let $\cl V = \frak{M}_{\max}(\kappa)$, $\cl B = \frak{M}_{\max}(\kappa')$
and $\cl U = \frak{M}_{\max}(\lambda)$. 
It is straightforward to check that 
the support of $\overline{\cl V + \cl B}$ is $\kappa\cup\kappa'$. 
By \cite[Corollary 4.2]{eletod}, $\kappa\cup\kappa'$ is operator synthetic, 
and hence $\frak{M}_{\max}(\kappa\cup\kappa') = \overline{\cl V + \cl B}$. 
By Proposition \ref{1234}, 
$$\cl F(\frak{M}_{\max}(\kappa),\frak{M}_{\max}(\lambda)) 
= \frak{M}_{\max}(\kappa)\bar\otimes \frak{M}_{\max}(\lambda).$$
By Theorem \ref{th_sump}, 
$$\cl F(\frak{M}_{\max}(\kappa\cup\kappa'),\frak{M}_{\max}(\lambda)) = 
\frak{M}_{\max}(\kappa\cup\kappa')\bar\otimes \frak{M}_{\max}(\lambda).$$
By Proposition \ref{1234}, $\rho((\kappa\cup\kappa')\times\lambda)$ is 
operator synthetic. 
\end{proof}

Our next aim is Theorem \ref{th_fws}, for whose proof we will need some auxiliary lemmas.

\begin{lemma}\label{l_refpres}
Let $\cl U\subseteq \cl B(H_1,H_2)$ be a masa-bimodule and $\phi$ be a Schur idempotent 
acting on 
$\cl B(H_1,H_2)$. Then
$\phi(\Ref{\cl U}) = \Ref{\phi(\cl U)} = \ran\phi\cap\Ref{\cl U}$.
\end{lemma}
\begin{proof}
By \cite[Proposition 3.3]{eletod}, 
$\Ref{\cl U}$ coincides with the space of all operators 
$X\in \cl B(H_1,H_2)$ such that $\psi(X) = 0$ whenever $\psi$  is a Schur idempotent annihilating $\cl U$.
Fix $T\in \Ref{\cl U}$ and 
let $\theta$ be a Schur idempotent on $\cl B(H_1,H_2)$ such that $\theta(\phi(\cl U)) = \{0\}$. 
Then $\theta\circ \phi(T) = 0$ and hence $\phi(T)\in \Ref{\phi(\cl U)}$; we thus showed that 
$\phi(\Ref{\cl U}) \subseteq \Ref{\phi(\cl U)}$.

Now suppose that $T\in \phi(\Ref{\cl U})$; then clearly $T\in \ran\phi$ and, by the previous paragraph, 
$T\in \Ref{\phi(\cl U)}\subseteq \Ref{\cl U}$. Thus, $\phi(\Ref{\cl U})\subseteq \ran\phi\cap\Ref{\cl U}$. 
On the other hand, if $T\in \ran\phi\cap\Ref{\cl U}$ then $T = \phi(T)\in \phi(\Ref{\cl U})$; hence,
$\phi(\Ref{\cl U}) = \ran\phi\cap\Ref{\cl U}$. 

By \cite[Proposition 3.3]{eletod}, $\ran\phi$ is reflexive and since 
reflexivity is preserved by intersections, 
the previous paragraph implies that $\phi(\Ref{\cl U})$ is reflexive.
Since $\phi(\cl U) \subseteq \phi(\Ref{\cl U})$, we have 
$\Ref{\phi(\cl U)}\subseteq \Ref{\phi(\Ref{\cl U})} = \phi(\Ref{\cl U})$, and the proof is complete. 
\end{proof}

\begin{lemma}\label{miden}
Let $\phi_i$ be a Schur idempotent,
$\kappa_i\subseteq X_1\times X_2$ be the support of $\ran\phi_i$, and
$\lambda_i \subseteq Y_1\times Y_2$ be an $\omega$-closed set, $i = 1,\dots,r$. Suppose that
$\cup_{k=1}^p \lambda_{m_k}$ is operator synthetic, whenever
$1\leq m_1 < m_2 < \cdots < m_p\leq r$. Then the set
$\rho(\cup_{i=1}^r \kappa_i\times \lambda_i)$ is operator synthetic.
\end{lemma}
\begin{proof}
Set $\kappa = \rho(\cup_{i=1}^r \kappa_i\times \lambda_i)$, 
$\cl U_i = \frak{M}_{\min}(\lambda_i)$ and 
$\cl W = \frak{M}_{\min}(\kappa)$. 
By (\ref{equv}), the support of the masa-bimodule
$\overline{\sum_{i=1}^r \frak{M}_{\min}(\kappa_i)\bar\otimes \frak{M}_{\min}(\lambda_i)}$ 
is $\rho(\cup_{i=1}^r \kappa_i\times\lambda_i)$; by the minimality property of $\cl W$
and the fact that the sets $\kappa_i$ and $\lambda_i$ are 
operator synthetic, we have that 
$$\cl W = \overline{\sum_{i=1}^r \frak{M}_{\min}(\kappa_i)\bar\otimes \frak{M}_{\min}(\lambda_i)} = 
\overline{\sum_{i=1}^r \frak{M}_{\max}(\kappa_i)\bar\otimes \frak{M}_{\max}(\lambda_i)}.$$

For each $i = 1,\dots,r$, let 
$\phi_i^1 = \phi_i$ and $\phi_i^{-1} = \phi_i^{\perp}$, 
and for each subset $M$ of $\{1,\dots,r\}$, let 
$\phi_M = \phi_1^{\epsilon_1} \phi_2^{\epsilon_2} \cdots \phi_r^{\epsilon_r},$
where $\epsilon_i = 1$ if $i\in M$ and $\epsilon_i = -1$ if $i\not\in M$. 

Fix $T\in \Ref{\cl W}$; we will show that $T\in \cl W$. This will then imply that 
$\cl W = \Ref{\cl W}$, and hence that 
$\rho(\cup_{i=1}^r \kappa_i\times \lambda_i)$ is operator synthetic.

We have $T = \sum \tilde{\phi}_M(T)$, where the sum is taken over all subsets $M$ of $\{1,\dots,r\}$. 
By Lemma \ref{l_refpres},
$\tilde{\phi}_M(T)\in \Ref{\tilde{\phi}_M(\cl W)}$ and hence
$$T\in \mbox{{\rm Ref}}\left(\sum_M \sum_{i=1}^r \tilde{\phi}_M(\ran\phi_i\bar\otimes\cl U_i)\right).$$
By Lemma \ref{l_pre}, 
$\tilde{\phi}_M(\ran\phi_i\bar\otimes\cl U_i) = \ran\phi_M\bar\otimes\cl U_i$
if $i\in M$, and $\tilde{\phi}_M(\ran\phi_i\bar\otimes\cl U_i) = \{0\}$ otherwise.
Thus, 
$$T\in \mbox{{\rm Ref}}\sum_M 
\left(\ran \phi_M \bar\otimes\overline{\sum_{i\in M} \cl U_i}\right).$$
We have that $\phi_M\phi_N = 0$ if $M\neq N$.
The assumption concerning the synthesis of the finite unions 
of the sets $\lambda_j$ implies that 
$\overline{\sum_{i\in M} \cl U_i} = \frak{M}_{\max}(\cup_{i\in M} \lambda_{i})$; 
we may thus assume that the maps $\phi_i$, $i = 1,\dots,r$ have the property
that $\phi_i\phi_j = 0$ if $i\neq j$.

We now proceed by induction on $r$. 
If $r = 1$, the statement follows from Lemma \ref{l_pre} and Corollary \ref{c_synps}.
Assume that the statement holds if the number of the given terms is at most $r-1$, and
recall that $T\in \Ref{\cl W}$. By Lemma \ref{l_refpres} and the inductive assumption, 
$$\tilde{\phi}_{r}^{\perp}(T) \in \Ref{\tilde{\phi}_{r}^{\perp}(\cl W)} \subseteq
\mbox{{\rm Ref}}\left(\sum_{i=1}^{r-1} \frak{M}_{\max}(\kappa_i)\bar\otimes \frak{M}_{\max}(\lambda_i)\right) \subseteq
\cl W.$$
On the other hand,
\begin{eqnarray*}
\tilde{\phi}_r(T) & \in & 
\Ref{\frak{M}_{\max}(\kappa_r)\bar\otimes \frak{M}_{\max}(\lambda_r)} 
= \frak{M}_{\max}(\kappa_r)\bar\otimes \frak{M}_{\max}(\lambda_r)\\
& = &
\frak{M}_{\min}(\kappa_r)\bar\otimes \frak{M}_{\min}(\lambda_r) \subseteq \cl W
\end{eqnarray*}
(we have used 
Lemma \ref{l_refpres} for the containment, and 
Corollary \ref{c_synps}, Proposition \ref{1234} and the fact that 
$\frak{M}_{\max}(\kappa_r)$ has property $S_{\sigma}$ for the first equality).
Thus,
$$T = \tilde{\phi}_r(T) + \tilde{\phi}_r^{\perp}(T) \in \cl W$$
and the proof is complete.
\end{proof}

\begin{lemma}\label{l_nestsy}
Let $\kappa_i\subseteq X_1\times X_2$ be the support of a nest algebra bimodule, and
let $\lambda_i \subseteq Y_1\times Y_2$ be an $\omega$-closed set, $i = 1,\dots,r$. Suppose that
$\cup_{k=1}^p \lambda_{m_k}$ is operator synthetic whenever
$1\leq m_1 < m_2 < \cdots < m_p\leq r$. Then the set
$\rho(\cup_{i=1}^r \kappa_i\times \lambda_i)$ is operator synthetic.
\end{lemma}

\begin{proof}
Set $\kappa = \rho(\cup_{i=1}^r \kappa_i\times \lambda_i)$.
Let $\cl B_i = \frak{M}_{\max}(\kappa_i)$, $\cl U_i = \frak{M}_{\max}(\lambda_i)$,
$i = 1,\dots,r$, and $\cl W = \overline{\sum_{i=1}^r \cl B_i\bar\otimes \cl U_i}$. 
As in the proof of Lemma \ref{miden},
$\cl W = \frak{M}_{\min}(\kappa)$ and hence $\Ref{\cl W} = \frak{M}_{\max}(\kappa)$.

Let $(\phi_{i,k},\theta_{i,k},\cl M_{i,k},\cl Z_{i,k})_{k\in \bb{N}}$ be a decomposition scheme 
for $\cl B_i$, $i = 1,\dots,r$. 
Set $\psi_{i,k} = (\phi_{i,k} + \theta_{i,k})^{\perp}$. 
For each subset $M$ of $\{1,\dots,r\}$, a subset $N$ of $M$, 
and indices $k_1,k_2,\dots,k_r\in \bb{N}$, 
we let 
$\gamma_{k_1,k_2,\dots,k_r}^{M,N} = \gamma_1\circ \dots\circ\gamma_r$, 
where 
$\gamma_i = \tilde{\phi}_{i,k_i}$ if $i\in N$,
$\gamma_i = \tilde{\theta}_{i,k_i}$ if $i\in M\setminus N$ and
$\gamma_i = \tilde{\psi}_{i,k_i}$ if $i\not\in M$. 
Fix $T\in \Ref{\cl W}$.
Then, by Lemmas \ref{l_pre} and \ref{l_refpres}, 
$$\gamma_{k_1,k_2,\dots,k_r}^{M,N}(T) \in 
\mbox{{\rm Ref}}\left( \sum_{i=1}^r \cl Y_{i,k_i}\bar\otimes \cl U_i \right),$$
where 
$\cl Y_{i,k_i}$ is equal to $\cl M_{i,k_i}$ if $i\in N$, to $\cl Z_{i,k_i}$ if $i\in M\setminus N$ and to 
$\{0\}$ if $i\not\in M$. 

Moreover, for all $k_1,k_2,\dots,k_r$, we have that
$$T = \sum_{M,N} \gamma_{k_1,k_2,\dots,k_r}^{M,N}(T),$$
where the sum is taken over all subsets $M$ and $N$of $\{1,\dots,n\}$ 
with $N\subseteq M$.
By Lemma \ref{miden}, 
$$\gamma_{k_1,k_2,\dots,k_r}^{M,N}(T)\in \overline{\sum_{i=1}^r \cl Y_{i,k_i}\bar\otimes \cl U_i}.$$
Since $\cl Z_{i,k}\subseteq \cl B_i$ for every $k\in \bb{N}$, we have that 
$$\gamma_{k_1,k_2,\dots,k_r}^{M,N}(T)\in \overline{\sum_{i=1}^r \cl X_{i,k_i}\bar\otimes \cl U_i},$$
where 
$\cl X_{i,k_i}$ is equal to $\cl M_{i,k_i}$ if $i\in N$, to $\cl B_i$ if $i\in M\setminus N$ and to 
$\{0\}$ if $i\not\in M$. 

Let $\{(M_p,N_p)\}_{p=1}^q$ be an enumeration of the 
pairs of sets $(M,N)$ with $N\subseteq M \subseteq \{1,\dots,r\}$.
For every fixed $r-1$-tuple $(k_1,\dots,k_{r-1})$ of indices, choose 
a weak* convergent subsequence 
$(\gamma_{k_1,k_2,\dots,k_{r-1},k_r'}^{M_1,N_1}(T))_{k_r'\in \bb{N}}$ of 
the sequence 
$(\gamma_{k_1,k_2,\dots,k_r}^{M_1,N_1}(T))_{k_r\in \bb{N}}$, and let 
$\gamma_{k_1,k_2,\dots,k_{r-1}}^{M_1,N_1}(T)$ be its limit.
Then choose 
a weak* convergent subsequence 
$(\gamma_{k_1,k_2,\dots,k_{r-1},k_r''}^{M_2,N_2}(T))_{k_r''\in \bb{N}}$ of 
the sequence 
$(\gamma_{k_1,k_2,\dots,k_{r-1},k_r'}^{M_1,N_1}(T))_{k_r'\in \bb{N}}$, and let 
$\gamma_{k_1,k_2,\dots,k_{r-1}}^{M_2,N_2}(T)$ be its limit.
Continuing inductively, define, for each pair $(M,N)$, an operator
$\gamma_{k_1,k_2,\dots,k_{r-1}}^{M,N}(T)$; 
by the choice of these operators, we have that
$$T = \sum_{M,N} \gamma_{k_1,k_2,\dots,k_{r-1}}^{M,N}(T).$$ 
By Proposition \ref{mazi} (v), 
$$\gamma_{k_1,k_2,\dots,k_{r-1}}^{M,N}(T) \in \overline{\sum_{i=1}^{r-1} \cl X_{i,k_i}\bar\otimes \cl U_i + 
\cl B_r\bar\otimes \cl U_r}.$$

We now choose, as in the previous paragraph, for every $r-2$-tuple 
$(k_1,\dots,k_{r-2})$ of indices, 
a weak* cluster point
$\gamma_{k_1,k_2,\dots,k_{r-2}}^{M,N}(T)$ of 
$(\gamma_{k_1,k_2,\dots,k_{r-1}}^{M,N}(T))_{k_{r-1}\in \bb{N}}$
such that 
$T = \sum_{M,N} \gamma_{k_1,k_2,\dots,k_{r-2}}^{M,N}(T)$, and use
Proposition \ref{mazi} (v) to conclude that
$$\gamma_{k_1,k_2,\dots,k_{r-2}}^{M,N}(T) \in \overline{\sum_{i=1}^{r-2} \cl X_{i,k_i}\bar\otimes \cl U_i + 
\cl B_{r-1}\bar\otimes \cl U_{r-1} + \cl B_r\bar\otimes \cl U_r}.$$

Continuing inductively, 
we conclude that $T = \sum_{M,N} \gamma_{\emptyset}^{M,N}(T)$, where 
$\gamma_{\emptyset}^{M,N}(T) \in \cl W$ for all subsets $N$ and $M$ of 
$\{1,\dots,r\}$ with $N\subseteq M$.
Hence, $T\in \cl W$ and the proof is complete.
\end{proof}

\begin{theorem}\label{th_fws}
Let $\kappa_i\subseteq X_1\times X_2$ be a set of finite width, and
let $\lambda_i \subseteq Y_1\times Y_2$ be an $\omega$-closed set, $i = 1,\dots,r$. Suppose that
$\cup_{k=1}^p \lambda_{m_k}$ is operator synthetic whenever
$1\leq m_1 < m_2 < \cdots < m_p\leq r$. Then the set
$\rho(\cup_{i=1}^r \kappa_i\times \lambda_i)$ is operator synthetic.
\end{theorem}
\begin{proof}
Let $\kappa = \rho(\cup_{i=1}^r \kappa_i\times \lambda_i)$,
$\cl B_i = \frak{M}_{\max}(\kappa_i)$ and write $\cl B_i = \cap_{j=1}^{l_i} \cl B_j^i$,
where $\cl B_j^i$ is a nest algebra bimodule, $i = 1,\dots,r$, $j = 1,\dots,l_i$.
Let also $\cl U_i = \frak{M}_{\max}(\lambda_i)$, $i = 1,\dots,r$.
Fix
$$T\in \frak{M}_{\max}(\kappa) = \mbox{{\rm Ref}}\left(\sum_{i=1}^r \cl B_i\bar\otimes\cl U_i\right).$$
Lemma \ref{l_nestsy} implies that, for all $j_1,\dots,j_r$, we have
$$T\in  
\mbox{{\rm Ref}}\left(\sum_{i=1}^r \cl B_{j_i}^i\bar\otimes\cl U_i\right) =
\overline{\sum_{i=1}^r \cl B_{j_i}^i\bar\otimes\cl U_i}.$$
By Corollary \ref{epsilona},
$$T\in \overline{\sum_{i=1}^r \cl B_i\bar\otimes\cl U_i} = \frak{M}_{\min}(\kappa).$$
\end{proof}

\begin{remark} 
{\rm 
In Theorem \ref{th_fws}, 
the condition that $\cup_{k=1}^p \lambda_{m_k}$ be operator synthetic whenever
$1\leq m_1 < m_2 < \cdots < m_p\leq r$ cannot be omitted.
Indeed, given such a choice of indices, 
fix a non-trivial subset of finite width $\kappa$, and let 
$\kappa_{m_j} = \kappa$, $j = 1,\dots,p$, and $\kappa_i = \emptyset$ if $i\not\in \{m_1,\dots,m_p\}$. 
If
$\rho(\cup_{i=1}^r \kappa_i\times\lambda_i) = \rho(\kappa\times(\cup_{j=1}^p \lambda_{m_j}))$ 
is operator synthetic then, by Remark \ref{r_indt}, 
$\cup_{j=1}^p \lambda_{m_j}$ is operator synthetic. }
\end{remark}

We conclude this section with an application of the previous results to 
spectral synthesis. 
Let $G$ be a second countable locally compact group.
By \cite{lt}, a closed set $E\subseteq G$ satisfies local
spectral synthesis if and only if the set 
$$E^* = \{(s,t)\in G\times G : st^{-1}\in E\}$$ is operator
synthetic (here $G$ is equipped with left Haar measure). 
We note that, in the
case the Fourier algebra $A(G)$ has an
approximate identity, $E$ is of local spectral synthesis if and only
if it satisfies spectral synthesis (see \cite{lt}).

Let $\bb{R}^+$ be the group of positive real numbers and
$\omega : G\rightarrow \bb{R}^+$ be a continuous group homomorphism.
For each $t > 0$, let
$$E_{\omega}^t = \{x\in G : \omega(x) \leq t\};$$
it is natural to call such a subset a \emph{level set}.
We have that
$$(E_{\omega}^t)^* = \{(x,y) \in G\times G : \omega(x) \leq t\omega(y)\}$$
and hence the intersections of the form
$$E = E_{\omega_1}^{t_1}\cap\dots\cap E_{\omega_k}^{t_k}$$
are a Harmonic Analysis version of sets of finite width: they have
the property that the corresponding set $E^*$ is a set of finite
width (see also \cite{eletod}).  
Theorem \ref{th_fws} has the following immediate consequence.

\begin{corollary}\label{c_ha}
Let $G$ and $H$ be second countable locally compact groups.
Suppose that $E_1,\dots,E_r$ are level sets in $G$ and $F_1,\dots,F_r$ are closed 
subsets of $H$ such that 
$\cup_{k=1}^p F_{m_k}$ is a set of local spectral synthesis whenever
$1\leq m_1 < m_2 < \cdots < m_p\leq r$.
Then the set $\cup_{i=1}^r E_i\times F_i$ is a set of local spectral 
synthesis of $G\times H$.
\end{corollary}


\section{Fubini products and Morita equivalence}\label{s_mor}

In this section, we show how tensor product formulas relate to the notion of 
spacial Morita equivalence introduced in \cite{eletro}.
For subspaces $\cl X$ and $\cl Y$ of operators, we let 
$$[\cl X\cl Y] = \overline{\left\{\sum_{i=1}^k X_iY_i : X_i\in \cl X, Y_i \in \cl Y, i = 1,\dots,k, k\in \bb{N}\right\}}.$$
We recall the following definition from \cite{ele3}:

\begin{definition}\label{E1} 
Let $\cl A$ (resp. $\cl B$ ) be a weak* closed unital algebra acting 
on a Hilbert space $H_1$ (resp. $H_2$). We say
that $\cl A$ is \emph{spatially embedded} in $\cl B$ if there exist a 
$\cl B,\cl A$-bimodule $\cl X\subseteq \cl B(H_1,H_2)$ and
an $\cl A,\cl B$-bimodule
$\cl Y\subseteq \cl B(H_2,H_1)$ such that
$[\cl X\cl Y] \subseteq \cl B$ and
$[\cl Y\cl X] = \cl A$.

If, moreover, $\cl B = [\cl X\cl Y]$, we call $\cl A$ and $\cl B$ \emph{spatially Morita equivalent}.
\end{definition}

We note
that if two unital dual operator algebras $\cl A$ and $\cl B$ are weak*  Morita equivalent in the sense of \cite{bk} then they have completely isometric representations 
$\alpha $ and $\beta $ such that the algebras
$\alpha (\cl A)$ and $\beta (\cl B)$ are spatially Morita equivalent
(this fact will not be used in the sequel).

\begin{theorem}\label{3.1} 
Let $H_1$, $H_2$ and $K$ be Hilbert spaces and $\cl A\subseteq \cl B(H_1)$ 
and $\cl B\subseteq \cl B(H_2)$
be weak* closed unital algebras.
Suppose that $\cl A$ is spatially embedded in $\cl B$ and
let $\cl U\subseteq \cl B(K)$ be a weak* closed space
such that $\cl F(\cl B, \cl U)= \cl B \bar\otimes  \cl U.$ 
Then $\cl F(\cl A, \cl U)=\cl A \bar\otimes \cl U.$
\end{theorem}

\begin{proof}
Let $\cl X\subseteq \cl B(H_1,H_2)$ and $\cl Y\subseteq \cl B(H_2,H_1)$ 
be subspaces satisfying the conditions of Definition \ref{E1}, and 
$T\in \cl A\bar\otimes \cl B(K)$ be such that
$R_\phi (T)\in \cl U$ for all
$\phi \in \cl B(H_1,H_2)_*.$ 
Suppose that $$T= \mbox{w}^*\mbox{-}\lim_n \sum_{i=1}^{m_n} A_i^n \otimes S_i^n,$$
where $(A_i^n)_{i=1}^{m_n} \subseteq \cl A$ and $(S_i^n)_{i=1}^{m_n} \subseteq \cl B(K).$
Fix $X_1\in \cl X$ and $Y_1\in \cl Y$ and set 
$S = (X_1\otimes I)T(Y_1\otimes I)$. 
For $\psi \in \cl B(H_2)_*$, let $\phi \in \cl B(H_1)_*$ be given by 
$\phi (A)=\psi (X_1AY_1)$, $A\in \cl B(H_1)$.
Since 
$$R_\phi (T) = \mbox{w}^*\mbox{-}\lim_n\sum_{i=1}^{m_n}\phi (A_i^n)S_i^n\in \cl U,$$ 
we have that $R_\psi (S)\in \cl U,$ for every $\psi\in \cl B(H_2)_*$. 
By our assumption, 
$S\in \cl B \bar\otimes \cl U$. 
Therefore,
$(X_1\otimes I)T(Y_1\otimes I)\in \cl B \bar\otimes  \cl U$ for all $X_1\in \cl X$ and all 
$Y_1\in \cl Y$. It follows that
$$(Y_2X_1\otimes I)T(Y_1X_2\otimes I)\in \cl A \overline{\otimes}  \cl U, \ \ \mbox{ for all }
X_1, X_2\in \cl X, Y_1, Y_2\in \cl Y.$$
Since $I\in \cl A = [\cl Y\cl X]$, it follows that $T\in \cl A\bar\otimes \cl U$.
\end{proof}

The following corollary is straightforward from Theorem \ref{3.1}.

\begin{corollary}\label{c_unal}
Suppose that $\cl A$ and $\cl B$ are weak* closed unital operator algebras.

(i) \ Suppose that $\cl A$ is spatially embedded in $\cl B$. 
If $\cl B$ has property $S_\sigma $ then so does $\cl A$.

(ii) Suppose that $\cl A$ and $\cl B$ are spatially Morita equivalent. 
Then $\cl A$ has property $S_\sigma $ precisely when $\cl B$ does so.
\end{corollary}

The last corollary, which is immediate from \cite{ele3} and Corollary \ref{c_unal},
concerns the inheritance of 
property $S_{\sigma}$ in the class of CSL algebras; we refer the reader to 
\cite{a} for the definition, relevant notation and theory of this class of algebras. 

\begin{corollary} 
Let $\cl L_1$ and $\cl L_2$ be CSL's.

(i) \ Suppose that $\phi : \cl L_1\rightarrow \cl L_2$ is a strongly continuous 
surjective lattice homomorphism. 
If $\mathrm{Alg}(\cl L_1) $  has property $S_\sigma $ then 
so does $\mathrm{Alg}(\cl L_2)$.

(ii) Suppose that $\phi : \cl L_1 \rightarrow \cl L_2$ is a strongly continuous 
lattice isomorphism.
Then the algebra $\mathrm{Alg}(\cl L_1) $  has property $S_\sigma $ if and only if 
$\mathrm{Alg}(\cl L_2)$ does so.
\end{corollary}

\end{document}